\documentclass[12pt]{amsart}
\usepackage{amssymb}
\usepackage{amscd}
\usepackage{amsfonts}
\usepackage{latexsym}
\usepackage[all]{xy}

\usepackage{verbatim}
\usepackage[margin=1.1in]{geometry}
\usepackage[hidelinks]{hyperref}

\newtheorem{theorem}{Theorem}[section]
\newtheorem{lemma}[theorem]{Lemma}
\newtheorem{proposition}[theorem]{Proposition}
\newtheorem{corollary}[theorem]{Corollary} 
\theoremstyle{definition}  
\newtheorem{definition}[theorem]{Definition}
\newtheorem{example}[theorem]{Example}
\newtheorem{conjecture}[theorem]{Conjecture}  
\newtheorem{question}[theorem]{Question}
\newtheorem{remark}[theorem]{Remark}

\newcommand{\id}{\text{id}}

\renewcommand{\Vec}{\operatorname{\operatorname{\mathsf{Vec}}}}
\DeclareMathOperator{\Pic}{\operatorname{\mathsf{Pic}}}
\DeclareMathOperator{\BrPic}{\operatorname{\mathsf{BrPic}}}

\DeclareMathOperator{\Aut}{\operatorname{\mathsf{Aut}}}
\DeclareMathOperator{\TAut}{\operatorname{\mathsf{TAut}}}
\DeclareMathOperator{\Out}{\operatorname{\mathsf{Out}}}
\DeclareMathOperator{\Inv}{\operatorname{\mathsf{Inv}}}

\DeclareMathOperator{\Rep}{\operatorname{\mathsf{Rep}}}
\DeclareMathOperator{\comod}{\operatorname{\mathsf{comod}}}
\DeclareMathOperator{\modd}{\operatorname{\mathsf{mod}}}
\DeclareMathOperator{\Ext}{\operatorname{\mathsf{Ext}}}

\DeclareMathOperator{\Lie}{\operatorname{\mathsf{Lie}}}
\DeclareMathOperator{\Hom}{\operatorname{\mathsf{Hom}}}

\newcommand{\rev}{\text{rev}}

\newcommand{\C}{\mathcal{C}}
\newcommand{\D}{\mathcal{D}}
\newcommand{\E}{\mathcal{E}}

\newcommand{\Z}{\mathcal{Z}}

\newcommand{\M}{\mathcal{M}}

\newcommand{\N}{\mathcal{N}}

\renewcommand{\O}{\mathcal{O}}

\newcommand{\be}{\mathbf{1}}

\newcommand{\g}{\mathfrak{g}}

\newcommand{\n}{\mathfrak{n}}

\renewcommand{\be}{\mathbf{1}}

\newcommand{\bt}{\boxtimes}
\newcommand{\ot}{\otimes}

\newcommand{\beq}{\begin{equation}}
\newcommand{\eeq}{\end{equation}}

\newcommand{\lb}{\label}
\newcommand{\bpf}{\begin{proof}}
\newcommand{\epf}{\end{proof}}

\newcommand{\bth}{\begin{theorem}}
\renewcommand{\eth}{\end{theorem}}
\newcommand{\bpr}{\begin{proposition}}
\newcommand{\epr}{\end{proposition}}
\newcommand{\ble}{\begin{lemma}}
\newcommand{\ele}{\end{lemma}}
\newcommand{\bco}{\begin{corollary}}
\newcommand{\eco}{\end{corollary}}
\newcommand{\bde}{\begin{definition}}
\newcommand{\ede}{\end{definition}}
\newcommand{\bex}{\begin{example}}
\newcommand{\eex}{\end{example}}
\newcommand{\bre}{\begin{remark}}
\newcommand{\ere}{\end{remark}}
\newcommand{\bcj}{\begin{conjecture}}
\newcommand{\ecj}{\end{conjecture}}

\newcommand{\End}{\text{End}}

\begin{document}
\title[Autoequivalences of tensor categories attached to quantum groups]{Autoequivalences of tensor categories attached to quantum groups at roots of $1$}
\date{\today}

\author{Alexei Davydov}
\address{A.D.: Department of Mathematics, Ohio University, Athens, OH 45701, USA}
\email{alexei1davydov@gmail.com}
\author{Pavel Etingof}
\address{P.E.: Department of Mathematics, Massachusetts Institute of Technology,
Cambridge, MA 02139, USA}
\email{etingof@math.mit.edu}
\author{Dmitri Nikshych}
\address{D.N.: Department of Mathematics and Statistics,
University of New Hampshire,  Durham, NH 03824, USA}
\email{dmitri.nikshych@unh.edu}

\begin{abstract}
We compute the  group of  braided tensor autoequivalences 
and the Brauer-Picard group of the representation category  
of the small quantum group $\mathfrak{u}_q(\mathfrak{g})$,
where $q$ is a root of unity.
\end{abstract}  

\maketitle

\centerline{\bf To the memory of Bertram Kostant} 

\baselineskip=18pt


\section{Introduction} 

Let $k$ be an algebraically closed field of characteristic zero. Let $G$ be a simple algebraic group over $k$, and let $\g=\Lie(G)$ be its Lie algebra. 
Let $q$ be a root of unity of odd order coprime to $3$ if $G$ is of type $G_2$, and 
coprime to the determinant of the Cartan matrix of $G$. 
Let $\mathfrak{u}_q(\g)$ be Lusztig's small quantum group attached to $\g$ \cite{Lu1}. Then $\mathfrak{u}_q(\g)$ is a quasitriangular Hopf algebra, so the category of its finite dimensional representations $\Rep\mathfrak{u}_q(\g)$ is a finite braided tensor category \cite{EGNO}.  One of the main goals of this paper is to compute the Picard group  $\Pic(\Rep\mathfrak{u}_q(\g))$ of this category, i.e., the group of equivalence classes of invertible $\Rep\mathfrak{u}_q(\g)$-module categories. Picard groups of braided tensor categories and, in particular, Brauer-Picard groups
of tensor categories play a crucial role in classification of graded extensions \cite{ENO} and also appear as symmetry groups
of three-dimensional topological field theories \cite{FS}.
It is known that $\Pic(\Rep\mathfrak{u}_q(\g))$ is isomorphic to the group $\Aut^{\rm br}(\Rep\mathfrak{u}_q(\g))$ of braided autoequivalences 
of $\Rep\mathfrak{u}_q(\g)$ \cite{DN, ENO}. We show under some restrictions on $q$ that $\Aut^{\rm br}(\Rep\mathfrak{u}_q(\g))$ is isomorphic to the group $\Aut(\g)$ of automorphisms of $\g$, i.e., $\Aut^{\rm br}(\Rep\mathfrak{u}_q(\g))=\Gamma\ltimes G^{\rm ad}$, where $G^{\rm ad}$ is the adjoint group of $G$ and $\Gamma=\Gamma_\g$ is the automorphism group of the Dynkin diagram of $\g$. Namely, we prove this when the order $\ell$ of $q$ is sufficiently large and also for classical groups $G=SL_N,Sp_N,SO_N$ if $\ell>N$. 

Moreover, we show that $\Rep\mathfrak{u}_q(\g)$ has only two braidings (the standard one and its reverse) and deduce that any tensor autoequivalence of $\Rep\mathfrak{u}_q(\g)$ is necessarily braided. Thus, the group of tensor autoequivalences (also known as the group of biGalois objects) of $\Rep\mathfrak{u}_q(\g)$ is isomorphic to $\Gamma\ltimes G^{\rm ad}$.  This generalizes the result of Bichon \cite{Bi2}, who proved this fact for $\g={\mathfrak{sl}}_2$.

We also consider the braided tensor category $\O_q(G)-\comod$ of finite dimensional comodules over the function algebra $\O_q(G)$, which is the $G$-equivariantization of $\Rep\mathfrak{u}_q(\g)$. We show that every braided autoequivalence of $\O_q(G)-\comod$ comes from a Dynkin diagram automorphism if $\ell$ is sufficiently large, and prove a similar result in the non-braided case. 
This generalizes a result of Neshveyev and Tuset \cite{NT1,NT2}, who proved this when $q$ is not a root of unity. 
We also show this for the classical groups $SL_N,Sp_N,SO_N$ if $\ell>N$. 

As a tool, we introduce the notion of a finitely dominated tensor category.
We show that the category of comodules over a finitely presented Hopf algebra is finitely dominated and 
prove that tensor autoequivalences of a finitely dominated category that preserve a tensor generator form an algebraic group. 
While this theory plays an auxiliary role in our paper, it may be of independent interest. 

We expect that the main results of this paper extend without significant difficulties to roots of unity of arbitrary order $\ell$, not necessarily satisfying the above coprimeness assumptions 
(at least when $\ell$ is sufficiently large). However, this would require some important modifications in the statements. 
Notably, if $\ell=2dr$, where $d=1,2,3$ is the the ratio of squared norms of long
and short roots of $G$, and $G$ is simply connected, then $\mathcal{O}_q(G)-\comod$ is a $G^L$-equivariantization (rather than a $G$-equivariantization)
of an appropriate version of $\Rep {\mathfrak{u}}_q(\g)$, where $G^L$ is the Langlands dual group 
of $G$, see \cite{AG}. Therefore, we expect that in this case $\Aut^{\rm br}(\Rep {\mathfrak{u}}_q(\g))\cong \Gamma\ltimes G^L$ (note that by definition $G^L=(G^L)^{\rm ad}$).   

The paper is organized as follows. In Section \ref{par} we give preliminaries and auxiliary results. In Section \ref{tetc} we develop the theory of finitely dominated tensor categories and study groups of tensor autoequivalences of such categories. In Section \ref{tao} we classify tensor autoequivalences of 
$\O_q(G)-\comod$. Finally, in Section \ref{tau} we prove that any tensor autoequivalence of $\Rep \mathfrak{u}_q(\g)$ is braided 
and classify such autoequivalences. As a consequence, we compute the Brauer-Picard groups of $\Rep \mathfrak{u}_q(\g)$ and 
$\Rep \mathfrak{u}_q(\mathfrak{b})$. 

{\bf Acknowledgements.} We are very grateful to Julien Bichon, 
Ken Brown, David Kazhdan, George Lusztig, Sergey Ne\-shvey\-ev, Victor Ostrik, Noah Snyder, and Milen Yakimov 
for useful discussions. The work of P.E.\ was partially supported by the NSF
grant DMS-1502244.  The work of D.N.\ was partially supported by the NSA grant H98230-16-1-0008.

\section{Preliminaries and auxiliary results} \label{par}

Let $k$ be an algebraically closed field. In this paper we consider tensor categories
over $k$ \cite[Definition 4.1.1]{EGNO} which we will simply refer to as {\em tensor categories}.  
The most basic example of a tensor category is the trivial tensor category $\Vec$, 
i.e., the category of finite dimensional vector spaces over $k$. More generally, 
given a Hopf $k$-algebra $H$, the category $\Rep H$ of finite dimensional left $H$-modules
and the category $H-\comod$ of finite dimensional left $H$-comodules 
are examples of tensor categories. For a tensor category $\C$, let $\Aut(\C)$
denote the group of isomorphism classes  of tensor autoequivalences of $\C$.

\subsection{Braided tensor categories and their Picard groups}
\label{prelim:braidings}

We refer the reader to \cite[Chapter 8]{EGNO} for basic definitions related to braided tensor categories.
Let $\C$ be a braided tensor category with braiding $c_{X,Y}:X\ot Y \to Y \ot X$. The {\em reverse}
braiding  of $c$ is, by definition, $c^{\rev}_{X,Y} := c_{Y,X}^{-1}$, $X,Y\in \C$.  We will denote by $\C^\rev$
the tensor category $\C$ with the reverse braiding.  Let $\Aut^{\rm br}(\C)$ denote the group
of isomorphism classes of braided tensor autoequivalences of $\C$.

Let $\Z(\C)$ denote the Drinfeld center of $\C$. 
Then the assignment $X \mapsto (X,\, c_{-,X})$ is a braided
embedding (i.e., a fully  faithful braided tensor functor) $\C \hookrightarrow \Z(\C)$.  
Similarly, the assignment $X \mapsto (X,\, c^{-1}_{X,-})$ is a braided embedding $\C^\rev \hookrightarrow \Z(\C)$. 
These embeddings combine together into a single braided tensor functor
\begin{equation}
\label{factorization map}
\C \bt \C^\rev \to \Z(\C).
\end{equation}

Assume in addition that $\C$ is finite (i.e., has finitely many simple objects  and enough projective objects). 
We say that $\C$ is {\em factorizable} if  the functor \eqref{factorization map} is an equivalence.

\begin{lemma}
\label{two and only two}
Let $\C$ be a factorizable braided tensor category with braiding $c$. 
Suppose that $\C$ is not pointed and that $\C$ has no  proper non-trivial tensor subcategories. 
Then $\C$ has exactly two braidings: $c$ and $c^{\rm rev}$. 
\end{lemma}
\begin{proof}
A braiding on $\C$ is the same thing as a section of the forgetful functor $F:\Z(\C)\to \C$, i.e.,
a tensor subcategory $\widetilde{\C}\subset \Z(\C)$ such that $F|_{\widetilde{\C}} : \widetilde{\C}\to \C$
is an equivalence.  By the hypothesis, $\widetilde{\C} \cap (\C\bt \Vec)$ is  either  $\C\bt \Vec$
or  $\Vec$. The former case 
corresponds to the original braiding $c$. 
Let us deal with the latter case. We will argue that $\widetilde{\C} \cap (\Vec\bt \C^\rev)$ is non-trivial, 
and, hence,  $\widetilde{\C} = \Vec\bt \C^\rev$, which corresponds to the reverse braiding $c^\rev$.

Any simple object of $\Z(\C)\cong \C\bt\C^\rev$  is of the form $X\bt Y$, where $X,\,Y$ are simple objects of $\C$. 
If $X\bt Y$ is in $\widetilde{\C}$ then $(X^*\ot X)\bt \be \in \widetilde{\C}$ since it is a subquotient of $(X\bt Y)^*\otimes (X\bt Y)$. 
Thus,  $(X^*\ot X)\bt \be \in \widetilde{\C}\cap (\C\bt \Vec)$, so $X$ must be invertible. Similarly,  $\be\bt (Y^*\ot Y) \in \widetilde{\C} \cap (\Vec\bt\C^\rev)$. Since $\widetilde{\C}$ is non-pointed,
one can choose a non-invertible $Y$. Therefore,  $\widetilde{\C} \cap (\Vec\bt \C^\rev) \neq \Vec$, as required.
\end{proof}

Let $\C$ be a finite braided tensor category. 
By definition, the {\em Picard group} $\Pic(\C)$ of  $\C$ \cite{ENO,DN} is the group of equivalence classes
of invertible $\C$-module categories.  When $\C$ is factorizable, there is a canonical group isomorphism
\begin{equation}
\label{Pic=Autbr}
\Pic(\C) \cong \Aut^{\rm br}(\C),
\end{equation}
see \cite[Theorem 6.2]{ENO} and \cite[Corollary 4.6]{DN}. 

Thus, computing the Picard group of a factorizable braided tensor category $\C$ amounts to computing
its group of braided tensor autoequivalences.

\begin{corollary}
\label{automatically braided}
Let $\C$ be a factorizable braided tensor category such that $\C \ncong\C^{\rm rev}$.  
Suppose that $\C$ is not pointed and that $\C$ has no  proper non-trivial tensor subcategories. 
Then any tensor autoequivalence of $\C$ is automatically braided and
\begin{equation}
\label{Pic=Aut}
\Pic(\C) \cong \Aut(\C).
\end{equation}
\end{corollary}
\begin{proof}
Follows from Lemma~\ref{two and only two}.
\end{proof}


\subsection{Algebraic group actions on categories, equivariantization and de-equ\-iva\-rian\-ti\-za\-tion} 
 
Let us recall the construction of \cite{AG}. Let $G$ be an abstract
group.  An action of $G$ on a category $\C$ \cite[2.7]{EGNO} is a
collection of functors $T_g: \C\to \C$ attached to each $g\in G$ such
that $T_1={\rm Id}_\C$, equipped with functorial isomorphisms
$\gamma_{g,h}: T_g\circ T_h\to T_{gh}$ satisfying the 2-cocycle
condition
$\gamma_{f,gh}\circ T_f(\gamma_{g,h})=\gamma_{fg,h}\circ
\gamma_{f,g}$.
Given such an action, a $G$-equivariant object in $\C$ is an object
$X$ together with a collection of isomorphisms $u_g: T_g(X)\to X$ such
that $u_{gh}\circ \gamma_{g,h}=u_g\circ T_g(u_h)$, and the
$G$-equivariantization $\C^G$ of $\C$ is the category of
$G$-equivariant objects in $\C$. If $\C$ is monoidal or braided, then
we require that the functors $T_g$ be monoidal, respectively braided,
and $\gamma_{g,h}$ preserve the tensor (respectively, braided) structure, 
in which case the equivariantization $\C^G$ inherits the same
structure.

For an affine group scheme $G$ over $k$, let $\O(G)$ denote the algebra of regular functions on $G$ 
(i.e., the coordinate Hopf algebra), and let $\O(G)-\modd$ denote the category of $\O(G)$-modules.
If $\C$ is artinian and $G$ is finite, then we can represent the collection of functors $\lbrace{T_g\rbrace}$ as a single functor 
$T: \C\to \O(G)-\modd\boxtimes \C$, where the Deligne tensor product $\O(G)-\modd\boxtimes \C$ may be interpreted as the category of $\O(G)$-modules in 
$\C$. Namely, $T_g(X)$ is the fiber of $T(X)$ at $g\in G$, and $T(X)=\oplus_{g\in G}T_g(X)$ (note that $\O(G)-\modd\cong \Vec_G$
in the case of finite $G$). Then the isomorphisms $\gamma_{g,h}$ are also combined into a single isomorphism $\gamma: ({\rm Id}\otimes T)\circ T\cong (m^*\otimes {\rm Id})\circ T$, where 
$$
m^*: \Vec_G=\O(G)-\modd \to \O(G)-\modd\boxtimes \O(G)-\modd=\O(G)\otimes \O(G)-\modd=\Vec_{G\times G}
$$ 
is the functor of sheaf-theoretic pullback under the multiplication map $m: G\times G\to G$: 
$m^*(M)_{g,h}=M_{gh}$.  Similarly, the morphisms $u_g$ are combined into a single morphism 
$u: T(X)\to T_{\rm triv}(X)$, where $T_{\rm triv}$ is the functor attached to the trivial $G$-action on $\C$. 
If $\C$ is monoidal or braided, we require that $T$ be a tensor (respectively, braided) functor, where 
the tensor product in $\O(G)-\modd\boxtimes \C$ is over $\mathcal{O}(G)$, and that $\gamma$ preserve these structures
in an appropriate sense. 

In this form, the definitions of an action and equivariantization make sense when $G$ is an {\it affine algebraic} group
(as they formalize the requirements that the functors $T_g$ and morphisms $\gamma_{g,h}$ depend algebraically on the group elements). 
Namely, in this case $\O(G)$ stands for the algebra of regular functions on $G$, $\Vec_G$ is replaced by the category of quasicoherent 
sheaves ${\rm QCoh}_G$, and $m^*(M)=\O(G\times G)\otimes_{\O(G)}M$, where $\O(G)$ embeds into $\O(G\times G)=\O(G)\otimes \O(G)$ via 
the coproduct (induced by the product in $G$). Also, $T(X)$ should be required to be an 
$\O(G)$-module in the ${\rm ind}$-completion of $\C$, rather than in $\C$ itself, and it is no longer the direct sum of $T_g$
(since $\O(G)$ may be infinite dimensional and non-semisimple). With these definitions, if $\C$ is a finite tensor category then $\C^G$ is a tensor category (in general, not finite), and if $\C$ is braided then so is $\C^G$ (provided that the $G$-action preserves the tensor structure and the braiding). Moreover, $\Rep G$ sits as a tensor subcategory in $\C^G$ (namely, the category of equivariant objects which are multiples to $\bold 1$ as objects of $\C$), and in the braided case this subcategory is contained in the M\"uger center of $\C^G$ (i.e., the squared braiding of an object of $\Rep G$ with any object of $\C^G$ is the identity).  

Finally, given a tensor category $\D$ together with a braided tensor functor $\Rep G\to {\mathcal{Z}}(\D)$ which gives rise to an inclusion $\Rep G\hookrightarrow \D$ (for an affine algebraic group $G$), we can define the de-equivariantization $\D_G$ of $\D$ to be the category of finitely generated 
$\O(G)$-modules in the ind-completion of $\D$, where $\O(G)$ is the algebra of regular functions on $G$ equipped with the action of $G$ by left (or, alternatively, right) translations \cite[8.23]{EGNO}. Moreover, if $\D$ is braided and $\Rep G$ lies in the M\"uger center of $\D$ then $\D_G$ is braided and carries a $G$-action,  and we have 
$(\D_G)^G\cong \D$.  Conversely, for a braided tensor category $\C$
with a $G$-action  we have $(\C^G)_G \cong \C$, i.e., equivariantization and de-equivariantization are inverses of each other
(this fact is essentially proved in  \cite{AG} and can also be obtained  by adjusting arguments of \cite[Section 4]{DGNO} to the infinite setting).
 
\subsection{Quantum groups at roots of unity}
\label{prelim:roots of 1}

Let ${\rm char}(k)=0$. Let $G$ be a simple algebraic group over $k$ and let $\g$ be the associated  Lie algebra. Let $\O(G)$ denote the coordinate Hopf algebra of $G$ and let $\O_q(G)$ denote its quantized form, see \cite{BG, KS}.

Let $\ell$  be an odd integer, relatively prime to the determinant of the Cartan matrix of $G$ and to $3$ if $\g$ is of type $G_2$. Let $q$ be a primitive $\ell$-th root 
of unity in $k$. Let $\mathfrak{u}_q(\mathfrak{g})$ be the small quantum group, i.e., the Frobenius-Lusztig kernel \cite{Lu1}. 
Recall \cite[XI.6.3]{T} that $\mathfrak{u}_q(\mathfrak{g})$ is a factorizable quasitriangular Hopf algebra. Also, it was shown in \cite{DL}
that there is  a cocleft central exact sequence of Hopf algebras
\begin{equation}
\label{cocleft sequence}
k \to \O(G) \to \O_q(G) \to \mathfrak{u}_q(\g)^* \to k\ .
\end{equation}
Moreover, the pullback of the coquasitriangular structure of $\mathfrak{u}_q(\g)^*$ (dual to the universal $R$-matrix of 
$\mathfrak{u}_q(\g)$) to $\O_q(G)$ defines a coquasitriangular structure on $\O_q(G)$, giving $\O_q(G)-\comod$ the structure of a braided tensor category such that the forgetful functor 
$\O_q(G)-\comod\to \mathfrak{u}_q(\g)^*-\comod=\Rep \mathfrak{u}_q(\g)$ is braided.

It follows that there is a natural action of $G$ on $\Rep\mathfrak{u}_q(\mathfrak{g}) = \mathfrak{u}_q(\g)^*-\comod$ as a braided category. Furthermore, it follows from \cite{AG, AGP} that with respect to this action $\O_q(G)-\comod$ is 
equivalent to the $G$-equivariantization of $\Rep \mathfrak{u}_q(\g)$, which can, in turn, 
be recovered as the de-equivariantization of $\O_q(G)-\comod$: 
 \begin{equation}
\label{OqG = equivariantization}
\O_q(G)-\comod = (\Rep \mathfrak{u}_q(\mathfrak{g}))^G, \quad  \Rep \mathfrak{u}_q(\mathfrak{g})
=(\O_q(G)-\comod)_G\ . 
\end{equation}
More precisely, \cite{AG} considers the case when $q$ is a root of unity of order $2dr$, where $d=1,2,3$ is the ratio of the squared norm of long roots of $G$ to the squared norm of short roots. In this case the role of $\O(G)$ is played by $\O(G^L)$, where $G^L$ is the Langlands dual group to $G$, but the arguments of \cite{AG} apply without significant changes to our case.   

It is easy to see that a maximal torus $T\subset G$ acts on $\Rep \mathfrak{u}_q(\g)$ by Hopf algebra automorphisms of $\mathfrak{u}_q(\g)$ (i.e., by conjugation). 
Hence, the center of $G$ acts trivially on $\Rep \mathfrak{u}_q(\g)$ (as it is contained in $T$, and conjugation by a central element induces the identity automorphism). 
On the other hand, the action of $G$ is non-trivial, as $\mathcal{O}_q(G)-\comod$ is not equivalent to $\Rep u_q(\g)\boxtimes \Rep G$. 
Thus, we get 

\begin{proposition}\label{inclus} 
There is an inclusion $G^{\rm ad}\hookrightarrow \Aut^{\rm br}(\Rep \mathfrak{u}_q(\g))$. 
\end{proposition} 

\begin{remark}
In the case of $\mathfrak{g}=\mathfrak{s}\mathfrak{l}_n$ the inclusion $G^{\rm ad}\hookrightarrow \Aut(\Rep \mathfrak{u}_q(\g))$
was established by Bichon \cite{Bi2}, who also proved that this inclusion is an isomorphism for $n=2$. 
\end{remark}

We also have the following (well known) lemmas. 

\begin{lemma}\label{noquot} ${\mathfrak{u}}_q(\g)$ has no nontrivial Hopf quotients. In particular, it has no nontrivial central grouplike elements. 
\end{lemma} 

\begin{proof}
The second statement follows since $\ell$ is coprime to the determinant of the Cartan matrix of $\g$.

To prove the first statement, recall that ${\mathfrak{u}}_q(\g)$ is a pointed Hopf algebra, and let $H:={\rm gr}\,{\mathfrak{u}}_q(\g)$ be its associated graded Hopf algebra under the coradical filtration (it is defined by the same generators $e_i,\, f_i,\, K_i$ and the same
relations as ${\mathfrak{u}}_q(\g)$, except that
now $[e_i,\, f_i]=0$ for all $i$). Then $H^*$ is a pointed Hopf algebra generated in degree $1$. Hence, any proper Hopf ideal $I\ne 0$ in ${\mathfrak{u}}_q(\g)$ must contain a nonzero element of degree $1$ under the coradical filtration. Hence $I$ must contain a nonzero $(1,g)$-skew-primitive element for some  grouplike element $g$. If it is trivial, i.e., is a multiple of $g-1$, then $I$ must contain $e_i$ for some $i$ (since $g$ is not central and thus acts on some $e_i$ with eigenvalue $\ne 1$). Thus, in any case $I$ contains a nontrivial $(1,g)$-skew-primitive element, say $e_i$. 
Since 
$$
[e_i,f_i]=\frac{K_i-K_i^{-1}}{q-q^{-1}},
$$ 
we have $K_i-K_i^{-1}\in I$. Thus $K_i-1\in I$ (as $\ell$ is odd). Thus $f_i\in I$ and $e_j,f_j\in I$ for $j$ connected to $I$ in the Dynkin diagram of $G$ (as $K_i$ acts on these elements with eigenvalues $\ne 1$). Continuing in this way, we will get that $e_i,f_i,K_i-1\in I$ for all $i$ (as the Dynkin diagram of $G$ is connected). Hence $I$ is the augmentation ideal. This implies the required statement. 
\end{proof} 

\begin{lemma}\label{noaut} 
The only tensor automorphism of the identity functor of $\Rep \mathfrak{u}_q(\g)$ is the identity. 
\end{lemma} 

\begin{proof} This follows from Lemma \ref{noquot}, since any such automorphism is defined by a central grouplike element. 
\end{proof} 

Finally, we will need the following lemma. Let $P$ be the weight lattice of $G$, $P_+$ its dominant part, and $L_\lambda$ be the simple $\O_q(G)$-comodule with highest weight $\lambda\in P_+$. Recall that $\O(G)-\comod=\Rep G\subset \O_q(G)-\comod$ is the semisimple subcategory whose simple objects are $L_{\ell \lambda}$ for $\lambda\in P_+$ (\cite{Lu2,DL}). 

\begin{lemma}\label{genera} 
If $\lambda$ is not divisible by $\ell$ then the matrix elements of 
$L_\lambda$ and $L_\lambda^*$ generate $\O_q(G/C_\lambda)$, where $C_\lambda$ is a central subgroup of $G$.   
\end{lemma} 

\begin{proof} Let $H\subset \O_q(G)$ be the Hopf subalgebra generated by the matrix elements of 
$L_\lambda$ and $L_\lambda^*$. It is clear that the $\O_q(G)$-comodule $L_{\ell\lambda}$ (which is a $G$-module) is a subquotient of $L_\lambda^{\otimes \ell}$. Let 
$C=C_\lambda\subset G$ be the kernel of the action of $G$ on $L_{\ell\lambda}$. Then the matrix elements of $L_{\ell\lambda}$ generate $\O(G/C)$, so $\O(G/C)\subset H$. 
Let $H_1$ be the fiber of the $\O(G/C)$-module $H$ at $1\in G/C$, a finite dimensional Hopf algebra. It is clear that $H\subset \O_q(G/C)$, and the fiber of $\O_q(G/C)$ at $1\in G/C$ is 
$\mathfrak{u}_q(\g)^*$, so $H_1$ is a Hopf subalgebra of $\mathfrak{u}_q(\g)^*$, hence $H_1^*$ is a Hopf quotient of 
$\mathfrak{u}_q(\g)$. Also $H_1\ne k$ since $\lambda$ is not divisible by $\ell$ and hence $H\ne \O(G/C)$. Since 
by Lemma \ref{noquot} $\mathfrak{u}_q(\g)$ has no nontrivial Hopf quotients, we get $H_1=\mathfrak{u}_q(\g)^*$. 
Thus, the de-equivariantization $(H-\comod)_{G/C}\subset (\mathcal{O}_q(G/C)-\comod)_{G/C}=\Rep \mathfrak{u}_q(\g)$
is actually the entire category $\Rep \mathfrak{u}_q(\g)$. Hence, $H-\comod=(\Rep \mathfrak{u}_q(\g))^{G/C}=\mathcal{O}_q(G/C)-\comod$ and $H=\O_q(G/C)$, as desired. 
\end{proof}   

\subsection{Compatibility of tensor functors on comodule categories with vector space dimensions.}
\begin{proposition}\label{compat} 
Let $H$ be a finitely generated Hopf algebra over $k$ of slower than exponential growth (e.g., of finite GK dimension).
Then 
\begin{enumerate}
\item[(i)] for any fiber functor $F: H-\comod\to \Vec$ one has $\dim F(X)\ge \dim X$; 
\item[(ii)] For any tensor autoequivalence $E: H-\comod\to H-\comod$ one has 
$\dim E(X)=\dim X$. 
\end{enumerate}
\end{proposition} 

\begin{proof} (i) We have ${\rm length}(X^{\otimes n})\le \dim F(X)^n$, so 
there is a simple composition factor $Y_n$ in $X^{\otimes n}$ which has dimension $d_n\ge (\dim X/\dim F(X))^n$. 
The matrix elements of $Y_n$ span a space of dimension $d_n^2$ and are noncommutative polynomials 
of degree $\le n$ of the matrix elements of $X$. So if $\dim X/\dim F(X)>1$, then $H$ has exponential growth. 

(ii) Let $F$ be the usual fiber functor on $H-\comod$. Then $F\circ E$ is another fiber functor, so 
by (i) $\dim E(X)=\dim F(E(X))\ge \dim X$. Also, the same is true for $E^{-1}$. Hence, $\dim E(X)=\dim X$.   
\end{proof} 

\begin{remark}\label{r1} If $H$ has exponential growth then both parts of Proposition \ref{compat} may fail. Indeed, Bichon showed in \cite{Bi1} that for any integer $n\ge 2$ there exists a Hopf algebra 
$H_n$
such that $H_n-\comod=\Rep SL_2(k)$, so that the 2-dimensional irreducible $SL_2(k)$-module corresponds to an $n$-dimensional $H_n$-comodule (namely, $H_2=\O(SL_2(k))$ but $H_n$ has exponential growth for $n\ge 3$). Now the usual fiber functor $F$ of $\Rep SL_2(k)$ on $H_n-\comod$ for $n\ge 3$ gives a counterexample to (i), and the autoequivalence of the category $H_m\otimes H_n-\comod$, $m\ne n$ 
switching the factors gives a counterexample to (ii).  
\end{remark}

\subsection{Basic properties of tensor autoequivalences of $\O_q(G)-\comod$.}

Let $F$ be a tensor autoequivalence of $\O_q(G)-\comod$. 
In this section we prove some basic properties of $F$. 

Recall that  a tensor category $\C$ is said to be {\em tensor-generated} by its object $X$ if every object 
of $\C$ is a subquotient of a direct sum of tensor powers of $X$.

\begin{lemma}\label{actiononrepG} $F$ induces an autoequivalence of 
$\O(G)-\comod=\Rep G$.
\end{lemma}

\begin{proof} By Lemma \ref{genera}, if $Y\in \O_q(G)-\comod$ and $Y\notin \Rep G$ then $Y\oplus Y^*$ tensor-generates a nonsemisimple category. 
Hence, if $X\in \Rep G$ is a simple object then $F(X)\in \Rep G$ (as $F(X)\oplus F(X)^*$ tensor-generates a semisimple category). 
The same holds for $F^{-1}$. This implies the statement. 
\end{proof} 

By the results of \cite{NT2}, the restriction of $F$ to $\Rep G$ belongs to the group \linebreak $\Out G\ltimes H^2(Z_G^\vee,k^\times)$, where $\Out G$ is the group of outer automorphisms of $G$ and $Z_G$ is the center of $G$ (this uses the theorem of McMullen that any automorphism of the Grothendieck semiring of $\Rep G$ comes from an automorphism of $G$). On the other hand, the group  $\Out G\ltimes H^2(Z_G^\vee,k^\times)$ acts naturally on $\O_q(G)-\comod$. So composing $F$ with an element of this group if needed, we may assume that $F|_{\Rep G}\cong {\rm Id}$. 

\begin{remark} If $F$ is braided then another way to prove Lemma \ref{actiononrepG} is to note that $\Rep G$ is the M\"uger 
center of $\O_q(G)-\comod$ (since  ${\mathfrak{u}}_q(\g)$ is a factorizable Hopf algebra,  
$\Rep {\mathfrak{u}}_q(\g)$ is a factorizable braided tensor category and so its M\"uger 
center is trivial), hence must be preserved by $F$. 

Moreover, in this case by the uniqueness of a fiber functor of a Tannakian category \cite{DM}, 
$F|_{\Rep G}$ is given by an outer automorphism of $G$. Thus, by composing $F$ with such an automorphism if needed, we may assume 
that $F|_{\Rep G}\cong {\rm Id}$ (i.e., we do not have to use \cite{NT2}). 
\end{remark}  

\begin{proposition}\label{dimpres}
For any finite dimensional $\O_q(G)$-comodule $V$ we have 
$\dim F(V)=\dim V$. 
\end{proposition} 

\begin{proof} This follows from Proposition \ref{compat}(ii), since $\O_q(G)$ has GK dimension $\dim G$ (as it is module-finite over $\mathcal{O}(G)$). 
\end{proof} 

\begin{proposition} \label{pressim} If $F$ is a tensor autoequivalence of $\O_q(G)-\comod$ such that $F|_{\Rep G}={\rm Id}$ then 
$F(L)\cong L$ for each simple object
$L\in \O_q(G)-\comod$. 
\end{proposition} 

\begin{proof}
Let $T\subset G$ be a maximal torus, and $W=N(T)/T$ the Weyl group. 
The character map gives an isomorphism 
${\rm Gr}(\O_q(G)-\comod)\otimes k\cong k[T/W]$ (where Gr stands for the Grothendieck ring). 
Thus, $F$ defines an automorphism $F^*: T/W\to T/W$. 
Moreover, let $\phi_\ell: T/W\to T/W$ be the map defined by raising 
to power $\ell$ on $T$. Then, since $F$ acts trivially on $\Rep G$, we have 
$F^*\circ \phi_\ell=\phi_\ell$. Also $F^*(1)=1$ 
by Proposition \ref{dimpres}. But the map $\phi_\ell$ defines 
an automorphism of the formal neighborhood of $1$ in $T/W$.
Hence, $F^*$ acts trivially on the formal neighborhood of $1$ in $T/W$, and 
hence $F^*={\rm Id}$. Thus, for each $L$ the character of $F(L)$ equals 
the character of $L$. But characters of simple objects
are linearly independent, which implies that 
$F(L)\cong L$ for all simple objects $L$.  
\end{proof} 

\begin{proposition}\label{noanti} 
There are no tensor autoequivalences of $\O_q(G)-\comod$ which reverse the braiding. 
\end{proposition} 

\begin{proof} Suppose for the sake of contradiction that $F$ is a braiding-reversing autoequivalence. 
Then $F$ preserves the M\"uger center $\Rep G$, and we may assume without loss of generality that $F|_{\Rep G}\cong {\rm Id}$. 
Hence, by Proposition \ref{pressim}, $F(L)\cong L$ for all simple objects~$L$. 

Let $P$ be the weight lattice of $G$ and $P_+\subset P$ be the set of dominant integral weights. 
The eigenvalue of the Drinfeld central element $z$ (the double twist) on the simple comodule $L_\lambda$ of highest weight $\lambda\in P_+$ is 
$q^{2(\lambda,\lambda+2\rho)}$. Since $F(L_\lambda)=L_\lambda$ and 
$F$ reverses braiding, this eigenvalue must equal its reciprocal, 
so we must have $(\lambda,\lambda+2\rho)=0$ in $\mathbb{Z}/\ell$ for all $\lambda\in P_+$.
Subtracting these conditions for two weights $\lambda,\mu$, we get $(\lambda-\mu,\lambda+\mu+2\rho)=0$ in $\mathbb{Z}/\ell$. 
Thus, $(\nu,\beta)=0$ in $\mathbb{Z}/\ell$ for all $\nu,\beta\in P$, which is a contradiction.      
\end{proof}

\section{Tensor autoequivalences of tensor categories} \label{tetc}

\subsection{Tensor autoequivalences of a finite tensor category} 

One of the goals of this section is to put algebraic structure on the groups $\Aut(\C)$ and their subgroups. 
We start with the following proposition. 

\begin{proposition}\label{autproalg} Let $\C$ be a finite tensor category over $k$. 
Then $\Aut(\C)$ has a natural structure of an affine algebraic group over $k$.  
Moreover, if $\C$ is braided then so does $\Aut^{\rm br}(\C)$. 
\end{proposition}

\begin{proof} The idea of the proof is to express categorical data 
(tensor functors) entirely in terms of linear-algebraic data (linear maps, i.e., eventually, matrices). 

Let $P$ be the direct sum of the indecomposable projectives of $\C$, and 
 $A:=(\End P)^{\rm op}$. Then we have a natural identification $\C\cong \Rep A$ as abelian categories, given by $Y\mapsto \Hom(P,Y)$. 
Under this identification, the tensor product functor $\otimes: \C\boxtimes \C\to \C$ is given by tensoring over $A^{\otimes 2}$ 
with an $(A,A^{\otimes 2})$-bimodule $T$, and the associativity isomorphism is represented by an isomorphism of $(A,A^{\otimes 3})$-bimodules 
$\Phi: T\otimes_{A^{\otimes 2}} (T\otimes A)\cong T\otimes_{A^{\otimes 2}} (A\otimes T)$
satisfying the pentagon relation. Any tensor autoequivalence $F:\C\to \C$ can then be defined 
by an algebra automorphism $\alpha: A\to A$ together with a bimodule isomorphism 
$J: T\cong T^\alpha$ which preserves $\Phi$. It is clear that pairs $(\alpha,J)$ 
form an affine algebraic group under the obvious composition. Denote this group by $G_1$. 
Also, let $G_2:=A^\times$, also an affine algebraic group. Then we have a homomorphism of algebraic groups $\phi: G_2\to G_1$ 
given by $\phi(a)=({\rm Ad}a,J_a)$, where $J_a(t)=at(a^{-1}\otimes a^{-1})$, $t\in T$.  
It is clear that a tensor functor $F$ determined by $(\alpha,J)$ is isomorphic to the identity if and only if $(\alpha,J)=\phi(a)$ for some $a\in A^\times$.
Thus, $\phi(G_2)$ is normal in $G_1$, and $\Aut(\C)=G_1/\phi(G_2)$ is an affine algebraic group, as claimed.  
\end{proof} 

\begin{remark} A different (but similar) proof of Proposition \ref{autproalg} may be obtained by using \cite[Proposition 2.7]{EO} which states that any finite tensor category $\C$ is the representation category of a finite dimensional weak quasi-Hopf algebra $H$, and representing tensor autoequivalences of $\C$ linear-algebraically as twisted automorphisms of $H$ (as in \cite{Da}). 
\end{remark} 

\subsection{Tensor autoequivalences of a tensor category generated by one object} 

Now let $\C$ be a tensor category which is not necessarily finite. Then 
in general $\Aut(\C)$ is not an algebraic or even a proalgebraic group. For instance, 
if $\C=\Vec_{\mathbb{Z}^2}$ then $\Aut(\C)=GL_2(\mathbb{Z})\ltimes k^\times$.  
Thus, to obtain a proalgebraic group, we need to put some restrictions on the  
tensor autoequivalences.  

Given $X\in \C$, let $\Aut_X(\C)$ be the subgroup of elements $F\in \Aut(\C)$ such that $F(X)\cong X$. 
The following proposition is a generalization of Proposition \ref{autproalg} to the infinite case. 

\begin{proposition}\label{genX} 
Suppose that $\C$ is tensor-generated by $X$.  Then $\Aut_X(\C)$ has a natural structure of an affine proalgebraic group.  Moreover, if $\C$ is braided then so does $\Aut_X^{\rm br}(\C)$.
\end{proposition} 
\begin{proof} 
The proof is analogous to the proof of Proposition \ref{autproalg}, with additional technical details 
to deal with the fact that $\C$ may not be finite.  

Namely, $\C$ has an exhausting increasing filtration $\C=\cup_{N\ge 0}\C_N$, where $\C_N$ is the full subcategory whose objects are subquotients of $(X\oplus \bold 1)^{\otimes N}$. Note that $\C_N$ are finite, and we have compatible tensor product functors 
$\otimes: \C_m\times \C_p\to \C_{m+p}$. Also, if $F: \C\to \C$ is a tensor autoequivalence such that $F(X)\cong X$ then 
$F$ preserves this filtration and these tensor products, and conversely, any autoequivalence $F: \C\to \C$ with these properties 
is a tensor autoequivalence such that $F(X)\cong X$ (indeed, $F(X\oplus \bold 1)\cong X\oplus \bold 1$ implies $F(X)\cong X$). 
Here by ``$F$ preserves the tensor products" we mean that $F$ is equipped with an appropriate tensor structure. 

Let $\Aut_X^N(\C)$ be the group of isomorphism classes of autoequivalences $\C_N\to \C_N$ preserving the filtration and the tensor  products 
$\otimes: \C_m\times \C_p\to \C_{m+p}$ for $m+p\le N$. Then the groups $\Aut_X^N(\C)$ form an inverse system (i.e., we have natural homomorphisms $\psi_N: \Aut_X^{N+1}(\C)\to \Aut_X^N(\C)$), and 
 $\Aut_X(\C)=\underleftarrow{\lim}_{N\to \infty}\Aut_X^N(\C)$. Thus, it suffices to show that $\Aut_X^N(\C)$ are affine algebraic groups
 and $\psi_N$ are homomorphisms of algebraic groups. 
 
 Let $P_N$ be the direct sum of the indecomposable projectives of $\C_N$. Then we have surjections 
 $\eta_N: P_N\to P_{N-1}$, such that ${\rm Ker} \eta_N$ is the intersection of kernels of all morphisms 
 from $P_N$ to objects of $\C_{N-1}$. Let $A_N:=(\End P_N)^{\rm op}$. Then any $a\in A_N$ preserves ${\rm Ker}\eta_N$ and thus descends to $P_{N-1}$, which defines morphisms $\xi_N: A_N\to A_{N-1}$. For $Y\in \C_{N-1}$, the map $\eta_N^*: \Hom(P_{N-1},Y)\to \Hom(P_N,Y)$ 
 is an isomorphism, which implies that  $\xi_N$ are surjective. Moreover, we have natural identifications $\C_N\cong \Rep A_N$ as abelian categories, given by $Y\mapsto \Hom(P_N,Y)$, and the inclusions $\C_{N-1}\hookrightarrow \C_N$ correspond to the surjections $\xi_N$. Under this identification, the tensor product functors $\otimes: \C_m\boxtimes \C_p\to \C_{m+p}$ are given by tensoring over $A_m\otimes A_p$ with an $(A_{m+p},A_m\otimes A_p)$-bimodule $T_{mp}$, 
 and the associativity isomorphism is represented by an isomorphism of $(A_{m+p+s},A_m\otimes A_p\otimes A_s)$-bimodules 
$$
\Phi_{mps}: T_{m+p,s}\otimes_{A_{m+p}\otimes A_s} (T_{mp}\otimes A_s)\cong T_{m,p+s}\otimes_{A_m\otimes A_{p+s}} (A_m\otimes T_{ps})
$$
satisfying the pentagon relation. Moreover, the bimodules $T_{mp}$ 
are equipped with natural identifications $\gamma_{m,p}^+: T_{mp}\otimes_{A_m} A_{m-1}\cong \xi_{m+p}^*T_{m-1,p}$,
$\gamma_{m,p}^-:T_{m,p-1}\otimes_{A_p} A_{p-1}\cong \xi_{m+p}^*T_{m,p-1}$, coming from restricting the tensor product 
$\otimes: \C_m\boxtimes \C_p\to \C_{m+p}$ to $\C_{m-1}\otimes \C_p$ and $\C_m\otimes \C_{p-1}$. 

Any autoequivalence $F:\C_N\to \C_N$ preserving $X$ and the tensor product functors can then be defined 
by a collection of algebra automorphisms $\alpha_m: A_m\to A_m$, $m\le N$ compatible with $\lbrace \xi_i\rbrace$ 
and a collection bimodule isomorphisms $J_{mp}: T_{mp}\cong T_{mp}^\alpha$, $m+p\le N$ preserving $\lbrace \Phi_{mps}\rbrace$ and compatible with $\lbrace\gamma_{mp}^\pm\rbrace$. It is clear that pairs $(\alpha,J)$ 
form an affine algebraic group under the obvious composition. Denote this group by $G_1$. 
Also, let $G_2:=A_N^\times$, also an affine algebraic group. Then we have a homomorphism of algebraic groups $\phi: G_2\to G_1$ 
given by $\phi(a)=({\rm Ad}a,J_a)$, where $J_a(t)=at(a^{-1}\otimes a^{-1})$, $t\in T_{mp}$.  
It is clear that a tensor functor $F$ determined by $(\alpha,J)$ is isomorphic to the identity if and only if $(\alpha,J)=\phi(a)$ for some 
$a\in A_N^\times$. Thus, $\phi(G_2)$ is normal in $G_1$, $\Aut_X^N(\C)=G_1/\phi(G_2)$ is an affine algebraic group, and $\psi_N$ is a morphism of algebraic groups, as claimed. 
\end{proof} 

\subsection{Tensor autoequivalences of the category of comodules over a Hopf algebra} 
Let us describe the affine proalgebraic group $\Aut_X(\C)$ more explicitly in the case when $\C=H-\comod$, where $H$ 
is a Hopf algebra over $k$. The condition that $\C$ is tensor-generated by $X$ means that $H$ is generated as an algebra by 
the finite dimensional subcoalgebra $C_X\subset H$ spanned by the matrix elements of $X$. 

\begin{definition} A {\it co-twisted automorphism} of $H$ is a pair $(\alpha,J)$, where $\alpha: H\to H$ is a coalgebra automorphism,
and $J\in (H\otimes H)^*$ is a Hopf 2-cocycle such that $\alpha(x*y)=\alpha(x)*_J\alpha(y)$, where $*_J$ is the product twisted by $J$.   
\end{definition} 

\begin{remark} This notion is dual to the notion of a twisted automorphism in \cite{Da}, which explains the terminology. 
\end{remark} 

Clearly, co-twisted automorphisms form a group under the obvious composition operation. Let us denote this group by $\TAut(H)$. 
We have a  natural homomorphism $\zeta: \TAut(H)\to \Aut(H-\comod)$ given by 
$\zeta(\alpha,J)=(F,J)$, where $F$ is defined by $\alpha$. 

Assume from now on that $H$ has slower than exponential growth. 

\begin{proposition}\label{ima} $\zeta$ is surjective.  
\end{proposition} 

\begin{proof} Let $(F,J)\in \Aut(H-\comod)$. 
Recall that $H=\oplus_V V_{\rm space}^*\otimes I(V)$, where $V$ runs over simple $H$-comodules and $I(V)$ is the injective hull of $V$
(here the subscript ``space'' indicates that we are considering $V^*$ just as a vector space, without any actions). 
 By Proposition \ref{compat}, $(F,J)$ preserves vector space dimensions of $H$-comodules.
Hence, $F(H)\cong H$, i.e., $F$ is induced by a coalgebra automorphism $\alpha$ of $H$. 
Further, since $(F,J)$ is a tensor equivalence, $J$ gives rise to a Hopf 2-cocycle on $H$ (which we denote by the same letter) such that $\alpha$ preserves the product up to twisting by $J$. Then $(\alpha,J)$ is a co-twisted automorphism of $H$  and $\zeta(\alpha,J)=(F,J)$. 
\end{proof} 

Let us now describe the kernel of $\zeta$. By definition, it consists of pairs $(\alpha,J)$ where $\alpha$ 
is the inner automorphism corresponding to an invertible element $a\in H^*$ and $J=da$ is the coboundary of $a$. 
Moreover, the pair $(\alpha,J)$ corresponding to  $a$ equals $(1,1)$ if and only if $a$ is a central grouplike element of $H^*$. 
Thus, ${\rm Ker}\zeta=(H^*)^\times/Z_H$, where $(H^*)^\times$ is the group of invertible elements of $H^*$ 
and $Z_H$ is the group of central grouplike elements of $H^*$. Therefore, we get 

\begin{corollary}\label{iso1} There is an isomorphism of groups
$$
\overline{\zeta}: \TAut(H)/((H^*)^\times/Z_H)\cong \Aut(H-\comod).
$$
\end{corollary} 

Let $\TAut_X(H)$ be the subgroup of co-twisted automorphisms $(\alpha,J)$ of $H$ for which $$\alpha^*(X)\cong X.$$ 
Also note that $(H^*)^\times$ is a proalgebraic group, and $Z_H$ is a closed normal subgroup in it. Thus, $(H^*)^\times/Z_H$ is a proalgebraic group, which sits as a closed normal subgroup in $\TAut_X(H)$. Therefore, we have

\begin{corollary}\label{iso2} The restriction $\zeta_X: \TAut_X(H)\to \Aut_X(H-\comod)$ is a surjective homomorphism 
of affine proalgebraic groups. Thus, we have an isomorphism of affine proalgebraic groups
$$
\overline{\zeta}: \TAut_X(H)/((H^*)^\times/Z_H)\cong \Aut_X(H-\comod).
$$
\end{corollary}  

\begin{proof}
This follows from Proposition \ref{genX}, Proposition \ref{ima}, and Corollary \ref{iso1}.   
\end{proof} 

\subsection{Finitely dominated tensor categories and their tensor autoequivalences}  

Let $\C$ be a tensor category, $\lbrace{V_i\rbrace}$ be a collection of objects of $\C$, and $\lbrace{ f_j\rbrace}$ be a collection of morphisms between tensor products of $V_i$. 

\begin{definition} We say that $\C$ is {\it dominated} by $\lbrace{V_i\rbrace}$ and $\lbrace{ f_j\rbrace}$ if any tensor functor $F$ from $\C$ to any tensor category $\E$
is determined by
$\lbrace{ F(V_i)\rbrace}$ and $\lbrace{ F(f_j)\rbrace}$ up to an isomorphism. We say that $\C$ is {\it finitely dominated} if it is dominated by a finite collection of objects and morphisms. 
\end{definition}

Here by $F(f_j)$ we mean the morphism between tensor products of $F(V_i)$ corresponding to $f_j$.  

\begin{remark} If $S\subset R$ are rings then one says that $a\in R$ is dominated by $S$ if for a ring homomorphism $f: R\to T$, $f(a)$ is determined by $f|_S$, \cite{Is}. Further, one says that $R$ is dominated by $S$ if this is true for each $a\in R$ (this is equivalent to saying that the inclusion $S\hookrightarrow R$ is an epimorphism in the category of rings). Note that this definition makes perfect sense if $S$ is just a subset of $R$, rather than a subring (in fact, the notions of domination by a subset $S$ and the subring $\langle S\rangle$ generated by $S$ are obviously equivalent). Thus, it is natural to talk of a ring $R$ dominated by a subset $S\subset R$. Note that 
this is a weaker condition than $R$ being generated by $S$ (e.g., $\Bbb Z[x,x^{-1}]$ is dominated by $x$). 
Our notion of domination is a categorification of this notion, which justifies the terminology. 
\end{remark}

\begin{proposition}\label{strogen}   
Suppose that $\C$ is finitely dominated  and tensor-generated by an object $X$. Then $\Aut_X(\C)$ is of finite type, i.e., is an affine algebraic group. 
\end{proposition}  

\begin{proof} Let $\lbrace{V_i,f_j\rbrace}$ be a finite set of objects and morphisms dominating $\C$.
Let $\C_N$ be the finite abelian subcategories of $\C$ defined in the proof of Proposition \ref{genX}. 
Let $N$ be so large that $V_i$ and $f_j$ belong to $\C_N$. 
Let $\Aut_X^N(\C)$ be the affine algebraic group defined in the proof of Proposition \ref{genX}, and let 
$K_N$ be the kernel of the natural homomorphism  $\Aut_X(\C)\to \Aut_X^N(\C)$. Then 
for $F\in K_N$, we have $F(V_i)=V_i$ for all $i$ and $F(f_j)=f_j$ for all $j$. Thus, $F\cong {\rm Id}$, i.e. $K_N=1$. 
Hence, $\Aut_X(\C)\subset \Aut_X^N(\C)$, so that $\Aut_X(\C)$ is an affine algebraic group, as claimed.  
\end{proof} 

\subsection{Finitely presented Hopf algebras} 

Let $H$ be a Hopf algebra which is finitely generated as an algebra. 
This means that the category $H-\comod$ is tensor-generated by a single object $X$. Hence,  we have a natural surjective homomorphism of $H$-bicomodule algebras 
$\theta: T(X\otimes {}^*X)\to H$. 

Now assume that $H$ is finitely presented (e.g., $H=\O(G)$, where $G$ is an algebraic group). This means that for some (and, hence, any) finite set of generators for $H$ there is a finite set of defining relations. In other words, the ideal ${\rm Ker}\theta$ is generated by a finite dimensional bicomodule $R\subset T(X\otimes {}^*X)$, so that $H$ is the cokernel of the natural $H$-bicomodule 
morphism 
$$
\psi: T(X\otimes {}^*X)\otimes R\otimes  T(X\otimes {}^*X)\to T(X\otimes {}^*X).
$$

We have $R\subset \oplus_{j=0}^m (X\otimes {}^*X)^{\otimes j}$. Let $\xi_j: R\to (X\otimes {}^*X)^{\otimes j}$ be the corresponding projections, $j=0,...,m$. 

\begin{proposition}\label{fin} 
Let $H$ be a finitely presented Hopf algebra and let $\E$ be a tensor category over $k$. Then a tensor functor 
$F: H-\comod\to \E$ is determined up to an isomorphism 
by $F(X)$, $F(R)$, and  $F(\xi_j): F(R)\to F(X)^{\otimes j}\otimes X_{\rm space}^{*\otimes j}$. 
That is,  the tensor category $H-\comod$ is finitely dominated (by $X,R$ and $\lbrace\xi_j\rbrace$). 
\end{proposition}

\begin{proof}  Let $\D$ be the diagrammatic additive monoidal category whose objects are direct sums of tensor products 
of symbols $X$ and $R$, and morphisms are freely generated by $\xi_j: R\to X^{\otimes j}\otimes X_{\rm space}^{*\otimes j}$, 
$j=0,\dots ,m$. Then we have a natural monoidal functor $$K: \D\to H-\comod,$$ and our job is to show that $F$ is determined up to an isomorphism by the composition $F\circ K: \D\to \E$.

In the completion $\widehat \D$ of $\D$ under infinite direct sums, there is an obvious morphism $\widetilde{\psi}$ such that 
$K(\widetilde{\psi})=\psi$. Hence, $F(\psi)=(F\circ K)(\tilde\psi)$ is determined by $F\circ K$. Thus, so is $F(H)=
{\rm Coker}F(\psi)$. Moreover, let $a: H\to H$ be a morphism of left $H$-comodules. Then $a$ is given by the right action of an element of $H^*$, which we will also call $a$. Let $\widetilde{a}$ be the endomorphism 
of $X^{\otimes j}\otimes X_{\rm space}^{*\otimes j}$ in $\D$ obtained by the right action of $a$ on the second component. 
By taking a direct sum over $j$, we get a morphism $$\widetilde{a}: T(X\otimes X_{\rm space}^*)\to T(X\otimes X_{\rm space}^*)$$ in $\widehat \D$
such that $\theta\circ K(\widetilde{a})=a\circ \theta$ in $H-\comod$. This implies that $F(a)$ is determined by $F\circ K$. 

Now, any $Y\in H-\comod$ admits an injective resolution $I_\bullet(Y)$ by multiples of $H$:
$$
Y\to I_0\to I_1\to \dots, 
$$
where $I_j=H\otimes W_j$, and $W_j$ is a vector space. Indeed, for any $H$-comodule $Z$ (not necessarily finite dimensional) we have a canonical inclusion $Z\hookrightarrow H\otimes Z_{\rm space}$; so we may define $I_j$ as $H\otimes (U_j)_{\rm space}$, where $U_j:={\rm Coker}(I_{j-2}\to I_{j-1})$ (where $I_{-1}:=Y$ and $I_0:=H\otimes Y_{\rm space}$). The above argument implies that the action of $F$ on this resolution is determined by $F\circ K$. Thus, $F(Y)$ is determined by $F\circ K$. 
Further, if $f: Z\to Y$ is a morphism in $H-\comod$ then it lifts to a morphism of resolutions $I_\bullet(Z)\to I_\bullet(Y)$, 
which implies that $F(f)$ is determined by $F\circ K$, i.e., $F$ is determined by $F\circ K$ as a functor. 

Now, the tensor structure $J_{ZY}: F(Z)\otimes F(Y)\to F(Z\otimes Y)$ is determined by its values at $Z=Y=H$.
In turn, $J_{HH}$ is determined by $J_{\widetilde{H}\widetilde{H}}$, where  $\widetilde{H}:=T(X\otimes X_{\rm space}^*)$.
Finally, $J_{\widetilde{H}\widetilde{H}}$ is determined by $F\circ K$.   
This proves the proposition. 
\end{proof} 

\begin{corollary}\label{co0}  Let $H$ be a finitely presented Hopf algebra such that $H-\comod$ is tensor-generated by $X$. Then 
\begin{enumerate}
\item[(i)] $\Aut_X(H-\comod)$ is an affine algebraic group; 
\item[(ii)] $\TAut_X(H)/((H^*)^\times/Z_H)$ is an affine algebraic group. 
\end{enumerate}
\end{corollary} 

\begin{proof} Part (i) follows from Proposition \ref{strogen} and Proposition \ref{fin}. 
Part (ii) follows from (i) and Corollary \ref{iso2}.
\end{proof} 

\subsection{The second invariant cohomology of  a tensor category}

Given a tensor category $\C$, let $\Aut_0(\C)$ be the group of isomorphism classes of tensor autoequivalences 
of $\C$ that are isomorphic to the identity as abelian equivalences, \cite{da, BC, GKa}. 
We will call this group the {\it second invariant cohomology group} of $\C$.  

\begin{proposition}\label{lcpag} $\Aut_0(\C)$ has a natural structure of an affine proalgebraic group. 
\end{proposition} 

\begin{proof} Let $S$ be the set of isomorphism classes of objects of $\C$. 
Let $G_1$ be the closed subgroup of the affine proalgebraic group
$\prod_{X,Y\in S}\Aut(X\otimes Y)$ cut out by the functoriality condition in $X$ and $Y$
and the tensor structure axiom 
$$
J_{X\otimes Y,Z}\circ (J_{X,Y}\otimes {\rm Id}_Z)=J_{X,Y\otimes Z}\circ ({\rm Id}_X\otimes J_{Y,Z}), 
$$
i.e., the closed subgroup in $\Aut(\otimes)$ cut out by the tensor structure axiom. 
Also let $G_2\subset \prod_{X\in S}\Aut(X)$ be the closed subgroup cut out by the functoriality condition in $X$ 
(i.e., $G_2=\Aut(\id_\C)$, a commutative affine proalgebraic group). Let $\phi: G_2\to G_1$ 
be the homomorphism defined by $\phi(a)_{X,Y}=a_{X\otimes Y}\circ (a_X^{-1}\otimes a_Y^{-1})$. 
Then $\phi$ is a homomorphism of affine proalgebraic groups, $\phi(G_2)$ is a central subgroup of $G_1$, 
and $\Aut_0(\C)={\rm Coker}\phi$. This implies the statement. 
\end{proof} 

\begin{remark} In the case $\C=H-\comod$, where $H$ is a Hopf algebra, Proposition \ref{lcpag} follows from Theorem 6.7 of \cite{BKa}.  
\end{remark} 

\begin{corollary}\label{co1} (i) If $\C$ is finitely dominated then the second invariant cohomology $\Aut_0(\C)$ is an affine algebraic group (i.e., is of finite type). 

(ii) If $H$ is a finitely presented Hopf algebra then the second invariant cohomology $\Aut_0(H-\comod)$ is  
an affine algebraic group. 
\end{corollary} 

\begin{proof} (i) It is clear that $\Aut_0(\C)$ is a closed subgroup of $\Aut_X(\C)$, where $X$ is a tensor-generating object. Thus, the statement follows from Proposition \ref{strogen}. 

Part (ii) follows from (i) and Corollary \ref{co0}. 
\end{proof} 

\begin{example} Let $G$ be an abstract group, and $H=kG$ be its group algebra. 
Then $H$ is finitely generated, respectively finitely presented, iff so is $G$.
Also, $H-\comod=\Vec_G$, the tensor category of finite dimensional $G$-graded vector spaces. Thus, $\Aut_0(H-\comod)=H^2(G,k^\times)$. 
Hence, Corollary \ref{co1} reduces in this case to the well known result that for a finitely presented group $G$ the group $H^2(G,k^\times)$ is an algebraic group, which holds due to the Serre-Hochschild exact sequence $\dots\to \Hom(N,k^\times)^G\to H^2(G,k^\times)\to 1,$ where $1\to N\to \widetilde{G}\to G\to 1$ is a finite presentation of $G$ (i.e., $\widetilde{G}$ is a finitely generated free group).
Indeed, since this presentation is finite, 
$N$ is the normal closure of a finitely generated group, hence 
$\Hom(N,k^\times)^G$ is an affine algebraic group (i.e., is of finite type). 

Note that if $G$ is finitely generated but not finitely presented, then this may be false. E.g., if $G$ is the Grigorchuk group \cite{Gri}, then $H_2(G,\mathbb{Z})$ is infinitely generated 2-elementary abelian, hence  $H^2(G, k^\times)=\Hom(H_2(G,\mathbb{Z}),k^\times)$ is an infinite product of copies of $\mu_2$ (a proalgebraic group of infinite type). 
Thus, the finite presentation assumption in Proposition \ref{fin} and the finite domination assumption in Proposition \ref{strogen} cannot be dropped. 
\end{example} 

\section{Tensor autoequivalences of $\O_q(G)-\comod$} \lb{tao}

\subsection{Finite presentation for $\O_v(G)$ over $k[v,v^{-1}]$}

Let $G$ be a simple algebraic group.
Let $v$ be an indeterminate and let $\O_v(G)$ denote the quantized function algebra of $G$ over $k[v,v^{-1}]$.

\begin{proposition}\label{finpre}  
$\O_v(G)$ is finitely presented as an algebra. 
\end{proposition}

\begin{proof} Let $S$ be a finite generating subset of
the cone $P_+$ of dominant weights of $G$. Let $X=\oplus_{\lambda\in S}V_\lambda$ be the direct sum of all standard $\O_v(G)$-comodules whose highest weights are in $S$. Then we have an element $T=T_X\in \End X\otimes \O_v(G)$ determining the coaction of $\O_v(G)$ on $X$.
By \cite[Proposition 3.3]{Lu3} the matrix elements of $T_X$ generate 
$\O_v(G)$. Also, $T_X$ satisfies the Faddeev-Reshetikhin-Takhtajan (FRT) relation
$$
R_{XX}^{12}T^{13}T^{23}=T^{23}T^{13}R_{XX}^{12},
$$  
where  $R_{XX}$ is the specialization of the universal R-matrix to $X\otimes X$ (note that to define $R_{XX}$, one may need to adjoin $v^{1/m}$ for some $m$, but the FRT relations contain only integer powers of $v$). Let $A$ be the algebra generated over $k[v,v^{-1}]$ by the matrix elements of $T$ with these relations taken as defining. Then we have a surjective homomorphism $\eta: A\to \O_v(G)$. 
By  \cite[Lemma I.8.17]{BG}, the algebra $A$ is Noetherian (more precisely, this lemma is proved when $v$ is specialized to 
$q\in k^\times$, but it generalizes verbatim to our setting). Hence the ideal ${\rm Ker}\eta$ is finitely generated. 
This implies that $\O_v(G)$ is finitely presented, as desired.  
\end{proof} 

\begin{corollary}\label{finpre1} For any $q\in k^\times$, the algebra $\O_q(G):=\O_v(G)/(v-q)$ is finitely presented.
Also, the $k(v)$-algebra $\widetilde{\O}_v(G):=\O_v(G)\otimes_{k[v,v^{-1}]}k(v)$ is finitely presented. 
\end{corollary}  

\begin{remark} 
\begin{enumerate}
\item[(1)] Proposition \ref{finpre} also holds over $\Bbb Z[v,v^{-1}]$ (i.e., in the setting of \cite{Lu3}), with the same proof. 
\item[(2)] A nice finite presentation of $\widetilde{\O}_v(G)$ (and thus of $\O_q(G)$ for transcendental $q$) 
is given in \cite{I}.   
\end{enumerate}
\end{remark} 

\subsection{Tensor autoequivalences of $\O_q(G)-\comod$ outside finitely many roots of unity} 

In this subsection we classify tensor and braided autoequivalences of $\O_q(G)-\comod$.
Here we don't make any coprimeness assumptions on the order of $q$, and just assume that $q$ is a root of unity of sufficiently large order $\ell$. 

Note that any tensor autoequivalence $F$ of ${\mathcal O}_q(G)-\comod$ naturally acts 
on the center $Z_G$ of $G$, as $Z_G^\vee$ is the universal grading group of ${\mathcal O}_q(G)-\comod$. Thus, for a subgroup $C\subset Z_G$, $F$ defines an equivalence 
${\mathcal O}_q(G/C)-\comod\to {\mathcal O}_q(G/F(C))-\comod$. 

Let $\Gamma_G=\Out G$ (e.g., if $G$ is simply connected then $\Gamma_G=\Gamma$) and \linebreak $\widetilde{\Gamma_G}:=\Out G\ltimes H^2(Z_G^\vee,k^\times)$.

\begin{theorem}\label{tensautoeq}  For all $q\in k^\times$ except finitely many roots of unity: 
\begin{enumerate}
\item[(i)] $\Aut(\O_q(G)-\comod)\cong \widetilde{\Gamma_G}$. 
\item[(ii)] $\Aut^{\rm br}(\O_q(G)-\comod)\cong \Gamma_G$.
\end{enumerate}
\end{theorem} 

\begin{proof} If $q$ is not a root of unity, this is shown in \cite{NT2} (for (i)) and \cite{NT1} (for (ii)); more precisely, the results of \cite{NT1,NT2} are proved for simply connected groups, but the arguments extend without significant changes to the general case. So we only have to prove the statements for roots of unity. 

Let us prove (i). For a positive integer $N$, let $\Sigma_N$ be the set of all nonzero dominant integral weights $\lambda$ for $G$ such that 
the irreducible representation $L_\lambda$ of $G$ with highest weight $\lambda$ 
has dimension $\le N$. If the order of $q$ is large enough, these 
$L_\lambda$ have $q$-analogs, $\O_q(G)$-comodules $L_\lambda^q$ of the same dimension as $L_\lambda$, which are also irreducible. 

Let $F\in \Aut(\O_q(G)-\comod)$.
We claim that for sufficiently large order $\ell$ of $q$, the functor $F$ permutes $L_\lambda^q$, $\lambda\in \Sigma_N$. 
Indeed, by Steinberg's tensor product theorem for quantum groups (\cite[Proposition 9.2]{Lu2}), for large enough $\ell$ the only irreducible comodules over $\O_q(G)$ 
which have dimension $\le N$, don't belong to the M\"uger center of \linebreak $\mathcal{O}_q(G)-\comod$, and cannot be nontrivially decomposed as a tensor product are $L_\lambda^q$, $\lambda\in \Sigma_N$. But $F(L_\lambda^q)$ cannot belong to the M\"uger center of  $\mathcal{O}_q(G)-\comod$, as it generates a subcategory of the form $\mathcal{O}_q(G/C)-\comod$, while $F(\mathcal{O}_q(G/C)-\comod)=\mathcal{O}_q(G/F(C))-\comod$ is not contained in the M\"uger center of $\mathcal{O}_q(G)-\comod$. So, since by Proposition \ref{dimpres} $F$ preserves vector space dimensions, it must permute $L_\lambda^q$, $\lambda\in \Sigma_N$. 

Now pick $N$ so large that $X=X_N^q:=\oplus_{\lambda\in \Sigma_N}L_\lambda^q$ is a tensor generator of $\O_q(G)-\comod$. Then $F(X)\cong X$, so $F\in \Aut_X(\O_q(G)-\comod)$. Note that we have a natural inclusion $\widetilde{\Gamma_G}\subset \Aut(\O_q(G)-\comod)$ (see \cite{NT2}). Thus, our job is to show that for sufficiently large $\ell$, this inclusion is an equality. 

Let $R=k[v,v^{-1}][1/f]$, where $f$ is a nonzero polynomial vanishing at roots of unity of low order.  
Since by Proposition \ref{finpre} $\O_v(G)$ is finitely presented, by Corollary \ref{co0} 
the commutative algebra $\O(\Aut_X(\O_q(G)-\comod))$ is the specialization 
at $v=q$ of a finitely generated commutative algebra $H$ over $R$ (for all but finitely many roots of unity $q$).
Indeed, we can take a finite presentation of the $k(v)$-algebra $\O(\Aut_X(\widetilde{\O}_v(G)-\comod))$ 
and define $H$ by the same generators and relations over $R$ (for a suitable choice of $f$). Moreover, since $\widetilde{\Gamma_G}$ acts faithfully by automorphisms of $\O_v(G)$, we have a surjective algebra homomorphism 
$\beta: H\to R[\widetilde{\Gamma_G}]$, where $R[\widetilde{\Gamma_G}]$ is the algebra of $R$-valued functions on $\widetilde{\Gamma_G}$. Let $K:={\rm Ker}\beta$. By Grothendieck's generic freeness lemma \cite[Theorem 14.4]{Eis}, since $H$ is finitely generated, we may assume without loss of generality that $H$ is a free $R$-module (by choosing $f$ appropriately). Then $H\cong K\oplus R[\widetilde{\Gamma_G}]$ as an $R$-module, hence $K$ is a projective $R$-module. 

Moreover,  by the result of \cite{NT2}, $\beta$ becomes an isomorphism upon tensoring with $k(v)$, hence $K\otimes_R k(v)=0$. 
As $K$ is projective, this implies that $K=0$, i.e., 
$\beta$ is an isomorphism. Thus, $\beta$ is an isomorphism after specializing $v$ to all roots of unity $q$ that are not roots of $f$.
Hence, for all such roots of unity $q$ we have an isomorphism $\Aut(\O_q(G)-\comod)\cong \widetilde{\Gamma_G}$, as desired.  

Part (ii) is proved in the same way, using the group $\Gamma_G$ instead of $\widetilde{\Gamma_G}$, and \cite{NT1} instead of \cite{NT2}. 
\end{proof}  

\begin{remark} It would be interesting to obtain a more direct proof of Theorem \ref{tensautoeq} (and desirably of its stronger version, 
giving an explicit list of excluded roots of unity) by generalizing the arguments  of \cite{NT1,NT2} to the case when $q$ is a root of unity. 
\end{remark} 

\subsection{Sharper results for classical groups} 

For classical groups $G=SL_N,Sp_N,SO_N$, we can use the Faddeev-Reshetikhin-Takhtajan presentations of $\O_q(G)$ to 
obtain a sharper result, i.e.,  one for all $q$ of order $\ell>N$. Note that in these cases $Z_G$ is cyclic, so 
$\widetilde{\Gamma_G}=\Gamma_G=\Out G$. 

\begin{theorem}\label{tensautoeq1} If $G=SL_N,Sp_N,SO_N$ and $\ell>N$ then 
$$\Aut(\O_q(G)-\comod)\cong \Aut^{\rm br}(\O_q(G)-\comod)\cong \Out G.$$   
\end{theorem} 

\begin{proof} Let us first prove that 
$\Aut^{\rm br}(\O_q(G)-\comod)\cong \Out G$. 
Take the tensor generator $X=V$, the defining comodule. 
In all three cases we have the Faddeev-Reshetikhin-Takhtajan presentation of $\O_q(G)$, in which $\O_q(G)$ is generated by the entries of $T\in \End V\otimes \O_q(G)$ with defining relations 
$$
R_{VV}^{12}T^{13}T^{23}=T^{23}T^{13}R_{VV}^{12} 
$$
plus some additional relations depending on which case we are considering, see \cite{FRT,RTF,Ta,Ha}.  

Consider first the case $G=SL_N$. Then the additional relation is the quantum determinant relation $\det_q(T)=1$. Thus, any braided autoequivalence of $\O_q(G)-\comod$ is determined by $F(V)$ and the action of $F$ on the morphism $\alpha: k\to V^{\otimes N}$ whose image is $\wedge^N_qV$, the quantum top exterior power of $V$. The only $N$-dimensional simple $\O_q(G)$-comodules are $V,V^*$ and their Frobenius twists $V^{(1)},(V^*)^{(1)}$ (since the set of weights of a comodule is Weyl group invariant). Then $F$ maps $V$ to $V$ or to $V^*$ (as $V\oplus V^*$ tensor-generates the category, while $V^{(1)}\oplus (V^*)^{(1)}$ does not). 
So by composing $F$ with an element of $\Out G=\mathbb{Z}/2$, 
we may assume without loss of generality that $F(V)=V$. By rescaling this isomorphism it is also easy to make sure that 
$F(\alpha)=\alpha$, so $F\cong {\rm Id}$, as desired. 

Now consider $G=Sp_{N}$ ($N$ even). Then the additional relation says that $T$ preserves the morphism $\beta : k\to V\otimes V$  
which deforms the symplectic form on $V^*$. Thus any braided autoequivalence $F$ of $\O_q(G)-\comod$ is determined by $F(V)$ and $F(\beta)$. As before, the only $N$-dimensional simple comodules are $V$ and the Frobenius twist $V^{(1)}$, and $F(V)\ncong V^{(1)}$ (since $V$ 
tensor-generates the category but $V^{(1)}$ does not). Thus $F(V)=V$. By rescaling this isomorphism we can also make sure 
that $F(\alpha)=\alpha$. Thus $F\cong {\rm Id}$, as desired. 

Finally, consider the case $G=SO_N$, $N\ge 3$.  
 In this case, the additional relations are that $T$ preserves the 
morphism $\beta: k\to V\otimes V$ which deforms the inner product on $V^*$, and that it preserves 
the morphism $\alpha: k\to V^{\otimes N}$ (i.e., has quantum determinant 1). Thus, any braided autoequivalence $F$ of $\O_q(G)-\comod$ is determined by $F(V)$, $F(\alpha)$, and $F(\beta)$. Moreover, the only $N$-dimensional simple
$\O_q(G)$-comodules are $V$ and $V^{(1)}$, and $F(V)\ne V^{(1)}$, since $V$ tensor-generates the category but $V^{(1)}$ does not. So, for any braided autoequivalence of $\O_q(G)-\comod$ we have 
$F(V)=V$.  Finally, we can rescale this isomorphism 
so that $F(\alpha)=\alpha$. 
Then $F(\beta)=\pm \beta$ since $\beta\otimes \beta$ may be expressed via $\alpha^{\otimes n}$ and hence 
$F(\beta)\otimes F(\beta)=\beta\otimes \beta$.  

Now we need to consider separately odd and even $N$.
If $N$ is odd, rescaling the isomorphism $F(V)\cong V$ by $-1$ (which preserves the relation $F(\alpha)=\alpha$), we can 
make sure that $F(\beta)=\beta$, so $F\cong {\rm Id}$. 
On the other hand, if $N$ is even, then we cannot do this, so we have two cases, $F(\beta)=\beta$ and $F(\beta)=-\beta$. 
But in this case we have a nontrivial involutive outer automorphism of $G$ implemented by an element of $O_N$ with determinant $-1$. So by composing with such automorphism, we can make sure that $F(\beta)=\beta$, i.e., $F\cong {\rm Id}$, as desired. 
\end{proof} 

\begin{remark} 1. Note that for $G=SO_8$, we have the group $S_3$ acting by Dynkin diagram automorphisms (hence automorphisms of ${\rm Spin}_8$), but only a 2-element subgroup of this $S_3$ descends to $SO_8$.  

2. The proof of Theorem \ref{tensautoeq1} is similar to the arguments of \cite{KW} for $SL_N$ and \cite{TW} 
for $SO_N$ and $Sp_N$. 
\end{remark} 

\section{Tensor autoequivalences of $\Rep \mathfrak{u}_q(\g)$} \label{tau}

Now let us classify tensor autoequivalences of $\Rep \mathfrak{u}_q(\g)$. We again assume  
that $q$ is a root of unity of odd order coprime to $3$ if $G$ is of type $G_2$, and 
coprime to the determinant of the Cartan matrix of $G$.

\subsection{The connected component of the identity of $\Aut(\Rep\mathfrak{u}_q({\mathfrak{g}}))$.}

Recall that for tensor categories  $\C,\,\C'$ and  a tensor functor $F:\C\to\C'$ 
one can define the deformation cohomology $H^i_F(\C)$, see \cite{da, Y} and \cite[Section 7.22]{EGNO}. Namely,
$C^i_F(\C)= \End(F\circ \otimes^i)$ with the usual differential, and $H^i_F(\C)$ is the $i$-th cohomology of the 
complex $C^i_F(\C)$. Then $H^2_F(\C)$ consists
of equivalence classes of first order deformations of $F$ as a tensor functor. Note that if $\C=\Rep(H)$, where $H$
is a Hopf algebra over $k$ and $F:\C\to \Vec$ is the forgetful functor, then $C^i_F(\C)=H^{\otimes i}$,
so if $H$ is finite dimensional then $H^i_F(\C)=H^i(H^*,\, k)$. In particular, for $H={\mathfrak{u}}_q(\g)$ we get 

\begin{proposition}\label{H2} One has 
\[
H^2_F(\Rep{\mathfrak{u}}_q(\g))= \n_+\oplus \n_-.
\]
where $\n_\pm$ are the positive and negative nilpotent subalgebras of $\g$.  
\end{proposition} 

\begin{proof} As explained above, we have 
\[
H^2_F(\Rep{\mathfrak{u}}_q(\g))=H^2({\mathfrak{u}}_q(\g)^*,\,k).
\]
On the other hand, by \cite[Proposition 2.3.1]{GK} and remark thereafter, we have 
\[
H^2({\mathfrak{u}}_q(\g)^*,\,k)= \n_+\oplus \n_-.
\]
This implies the statement.
\end{proof}  

Now we can compute the identity component of $\Aut(\Rep{\mathfrak{u}}_q(\g))$. 

\begin{proposition}\label{conncomid} 
\label{Pic0=G} One has $\Aut(\Rep{\mathfrak{u}}_q(\g))_0=\Aut^{\rm br}(\Rep{\mathfrak{u}}_q(\g))_0=G^{\rm ad}$. 
\end{proposition}

\begin{proof}
We have a natural map 
\begin{equation}
\label{Lie to H2}
\Lie \Aut(\Rep{\mathfrak{u}}_q(\g)) \to H^2_F(\Rep {\mathfrak{u}}_q(\g)).
\end{equation}
Namely, recall that elements of $\Aut(\Rep{\mathfrak{u}}_q(\g))$ are twisted automorphisms
$(a,\,J)$ of ${\mathfrak{u}}_q(\g)$ (in the sense of \cite{Da}),  so the 
map \eqref{Lie to H2} attaches to a twisted derivation $(d,\,j)$ the class of infinitesimal twist $j$.
(Here by a twisted derivation we mean a first order deformation of the identity twisted automorphism, see, e.g., \cite{Da3}).  Thus, by Proposition \ref{H2} we have a map
\[
\phi: \Lie\Aut(\Rep{\mathfrak{u}}_q(\g))\to \n_+\oplus \n_-. 
\]
The kernel of $\phi$ consists of  ordinary Hopf algebra
derivations of ${\mathfrak{u}}_q(\g)$, and it is easy to see that they can be identified with the Cartan subalgebra ${\mathfrak{t}}=\Lie(T)$ 
of $\g$. Thus, $\dim  \Lie\Aut(\Rep{\mathfrak{u}}_q(\g))\le \dim{\mathfrak t}+\dim \n_++\dim \n_-=\dim \g$. 
Since by Proposition \ref{inclus} we have an embedding 
$$
G^{\rm ad}\hookrightarrow \Aut^{\rm br}(\Rep{\mathfrak{u}}_q(\g))\subset \Aut(\Rep{\mathfrak{u}}_q(\g)),
$$
this implies the required statement. 
\end{proof}

\subsection{Tensor autoequivalences of $\Rep{\mathfrak{u}}_q(\g)$ are braided} 

\begin{proposition}
\label{not equivalent to reverse}
As braided tensor categories, $\Rep\mathfrak{u}_q(\mathfrak{g}) \ncong \Rep\mathfrak{u}_q(\mathfrak{g})^{\rm rev}$.
\end{proposition}

\begin{proof} Assume that $\overline{F}: \Rep\mathfrak{u}_q(\mathfrak{g}) \cong \Rep\mathfrak{u}_q(\mathfrak{g})^\rev$ is a braided equivalence.
Then $\overline{F}$ induces an automorphism $\gamma$ of $\Aut(\Rep{\mathfrak{u}}_q(\g))$ and in particular of its connected component 
of the identity $\Aut(\Rep{\mathfrak{u}}_q(\g))_0$, which by Proposition~\ref{conncomid} is $G^{\rm ad}$. Since $\Aut(G^{\rm ad})=\Gamma\ltimes G^{\rm ad}$ and every element of $\Gamma\ltimes G^{\rm ad}$ is implemented by a tensor autoequivalence in 
$\Aut(\Rep{\mathfrak{u}}_q(\g))$, by composing with such an autoequivalence, we may assume without loss of generality 
that $\gamma=1$. Then $\overline{F}$ commutes with $G^{\rm ad}$. Moreover, by Lemma \ref{noaut}, this commutativity 
is an isomorphism of actions. 
Hence, by Subsection \ref{prelim:roots of 1}, $\overline{F}$ gives rise to a braided equivalence of $G$-equivariantizations 
$F: \O_q(G)-\comod\to (\O_q(G)-\comod)^{\rm rev}$. But this contradicts Proposition \ref{noanti}.  
\end{proof} 

Now we finally obtain 

\begin{theorem} Every tensor autoequivalence of $\Rep \mathfrak{u}_q(\g)$ is automatically braided. 
In other words, we have $\Aut(\Rep{\mathfrak{u}}_q(\g))=\Aut^{\rm br}(\Rep{\mathfrak{u}}_q(\g))=\Pic(\Rep{\mathfrak{u}}_q(\g))$.
\end{theorem} 

\begin{proof} By Lemma \ref{noquot}, $\Rep{\mathfrak{u}}_q(\g)$
has no nontrivial tensor subcategories. By Proposition \ref{not equivalent to reverse}, 
the category $\C=\Rep{\mathfrak{u}}_q(\g)$ satisfies the assumptions of 
Corollary~\ref{automatically braided}. Thus, Corollary~\ref{automatically braided} implies the result. 
\end{proof} 

\subsection{Classification of tensor autoequivalences of $\Rep{\mathfrak{u}}_q(\g)$}
Introduce the notation 
$\mathbf{P}:=\Aut^{\rm br}(\Rep{\mathfrak{u}}_q(\g)).$ 
We have seen that $\mathbf{P}$ contains $\Gamma\ltimes G^{\rm ad}$, and by Proposition \ref{Pic0=G}
we have $\mathbf{P}_0=G^{\rm ad}$. Hence, $G^{\rm ad}$ is normal in $\mathbf{P}$. Given $x\in \mathbf{P}$, let $x'$ be the element of $\Gamma\ltimes G^{\rm ad}=\Aut(G^{\rm ad})$ induced by $x$. We can view $x'$ as an element of $\mathbf{P}$. Then $x=x'x''$, where $x''$ belongs to the centralizer $Z$ of $G^{\rm ad}$ in $\mathbf{P}$ (a finite group). Since $\Gamma$ normalizes $G^{\rm ad}$, it acts on $Z$ by conjugation. Thus, we have 

\begin{lemma}\label{structofP} 
$\mathbf{P}=\Gamma\ltimes (G^{\rm ad}\times Z)$.
\end{lemma}  

We can now formulate one of the main results of this paper. 

\begin{theorem}\label{tensauto} 
One has $\mathbf{P}=\Gamma\ltimes G^{\rm ad}$ in the following cases: 
\begin{enumerate}
\item[(i)] If $\g$ is of a classical type ($\mathfrak{sl}_N,\,\mathfrak{so}_N$, or $\mathfrak{sp}_N$) and  the order of $q$ is bigger than $N$; 
\item[(ii)] If $\g$ is exceptional and the order of $q$ is sufficiently large. 
\end{enumerate}
\end{theorem} 

\begin{proof} By Lemma \ref{structofP}, 
our job is to show $Z=1$. There is a group homomorphism from $Z\times G$ to $\Aut^{\rm br}(\Rep {\mathfrak{u}}_q(\g))$. 
By Lemma \ref{noaut}, this  homomorphism admits a unique lift to an action on $\Rep {\mathfrak{u}}_q(\g)$. 
Since the action of $Z$  on $\Rep{\mathfrak{u}}_q(\g)$ commutes with that of $G$, we conclude that $Z$  
acts on the equivariantization $(\Rep{\mathfrak{u}}_q(\g))^G$, 
which is the braided category $\O_q(G)-\comod$. 

Thus, it suffices to prove that the group of braided autoequivalences of $\O_q(G)-\comod$ coincides with $\Gamma_G$. 
Indeed, then given $z\in Z$, 
this would yield that $z\in \Gamma_G$, hence $z=1$ (as it acts trivially on $G^{\rm ad}$). 

Now part (i) follows from Theorem \ref{tensautoeq1} and part (ii) follows from Theorem \ref{tensautoeq}(ii). 
\end{proof} 

\begin{remark} We expect that Theorem \ref{tensauto} holds without the assumptions on the order of $q$. 
\end{remark} 

\subsection{Brauer-Picard groups} 

Let $\C$ be a finite tensor category.
Recall \cite{ENO} that the {\em Brauer-Picard group}  $\BrPic(\C)$  of $\C$ 
is the group of equivalence classes of invertible $\C$-bimodule categories. Recall also that there is a canonical isomorphism 
\begin{equation}\label{canoiso} 
\BrPic(\C)\cong \Aut^{\rm br}({\mathcal Z}(\C)),
\end{equation} 
see \cite{DN}  (and \cite{ENO} in the semisimple case). 

When $\C$ is braided, its Picard group  $\Pic(\C)$  is naturally identified with a subgroup of $\BrPic(\C)$.
 
\begin{proposition} 
Let $\C$ be a finite tensor category.
\begin{enumerate}
\item[(i)] The group $\BrPic(\C)$ has a natural structure of an affine algebraic group over $k$.
\item[(ii)] If $\C$ is braided then $\Pic(\C)$  has a natural structure of an affine algebraic group over~$k$.
\end{enumerate}
\end{proposition}  
\begin{proof}
Part (i)  follows immediately from \eqref{canoiso} and Proposition \ref{autproalg}.  To prove part (ii),
recall  that  under isomorphism \eqref{canoiso} $\Pic(\C)$ is identified with the subgroup of classes
of autoequivalences trivializable on the subcategory $\C \subset \Z(\C)$ (i.e., those $\alpha \in \Aut^{\rm br}(\Z(\C))$
for which  $\alpha|_\C \cong  {\rm Id}_\C$ as tensor functors), see \cite{DN}. 
This  means that $\Pic(\C)$ is a Zariski closed subgroup of $\BrPic(\C)$. 
\end{proof} 
 
In this subsection we will compute the Brauer-Picard groups of $\Rep \mathfrak{u}_q(\g)$ and 
$\Rep \mathfrak{u}_q(\mathfrak{b})$, where ${\mathfrak{b}}\subset \mathfrak{\g}$ is a Borel subalgebra. 

Let $E:=P/\ell P=Q/\ell Q$. Note that $E$ has a natural quadratic form $\mathbf{q}(v):=q^{(v,v)}$. Let $O(E,\mathbf{q})$ be the orthogonal group of this quadratic form. Note that $\Gamma$ acts naturally on $E$ preserving $\mathbf{q}$ and therefore  $\Gamma\hookrightarrow O(E,\mathbf{q})$. 

\begin{proposition}\label{BP} Under the assumptions of Theorem \ref{tensauto} one has: 
\begin{enumerate}
\item[(i)] $\BrPic(\Rep \mathfrak{u}_q(\mathfrak{b}))\cong (\Gamma\ltimes G^{\rm ad})\times O(E,\mathbf{q})$;
\item[(ii)] $\BrPic(\Rep \mathfrak{u}_q(\g))\cong (\Gamma\ltimes G^{\rm ad})\times (\Gamma\ltimes G^{\rm ad})$. 
\end{enumerate}
\end{proposition}  

\begin{proof} (i) It is well known that the quantum double $D(\mathfrak{u}_q(\mathfrak{b}))$ is given by 
 $$D(\mathfrak{u}_q(\mathfrak{b}))=\mathfrak{u}_q(\g)\otimes k[E],$$ where the R-matrix is the external product of the R-matrix of $\mathfrak{u}_q(\g)$ with the R-matrix on $k[E]$ defined by $\mathbf{q}$. Hence, 
${\mathcal Z}(\Rep\mathfrak{u}_q(\mathfrak{b}))$ is equivalent as a braided category to the category $\Rep \mathfrak{u}_q(\g)\boxtimes \Rep E$, where the braiding on the second factor is defined by $\mathbf{q}$. 
Thus, by \eqref{canoiso}, we have 
$$
\BrPic(\Rep \mathfrak{u}_q(\mathfrak{b}))\cong \Aut^{\rm br}(\Rep \mathfrak{u}_q(\g)\boxtimes \Rep E).
$$
Now, any braided autoequivalence $F$ of $\Rep \mathfrak{u}_q(\g)\boxtimes \Rep E$ must preserve the second factor, since it is the subcategory spanned by all the invertible objects. Hence $F$ also preserves the first factor (as it is the centralizer of the second one). Thus, we get 
$$
\BrPic(\Rep \mathfrak{u}_q(\mathfrak{b}))\cong \Aut^{\rm br}(\Rep \mathfrak{u}_q(\g))\times \Aut^{\rm br}(\Rep E).
$$
Since $\Aut^{\rm br}(\Rep E)\cong O(E,\mathbf{q})$, the result follows from Theorem \ref{tensauto}. 

(ii) Since the category $\C:=\Rep\mathfrak{u}_q(\g)$ is factorizable, one has
${\mathcal Z}(\C)\cong \C\boxtimes \C^{\rm rev}$. Thus, by \eqref{canoiso}, we have 
$$
\BrPic(\Rep \mathfrak{u}_q(\mathfrak{g}))\cong \Aut^{\rm br}(\C\boxtimes \C^{\rm rev}).
$$
It follows from Lemma \ref{noquot} that the only nontrivial tensor subcategories of 
$\C\boxtimes \C^{\rm rev}$ are $\C$ and $\C^{\rm rev}$, which are not braided equivalent
by Proposition \ref{not equivalent to reverse}. Hence, any braided autoequivalence of $\C\boxtimes \C^{\rm rev}$ must preserve both factors. 
So we get 
$$
\BrPic(\Rep \mathfrak{u}_q(\mathfrak{g}))\cong \Aut^{\rm br}(\C)\times \Aut^{\rm br}(\C),
$$
and the result follows from Theorem \ref{tensauto}.
\end{proof} 

Let $B\subset G$ be the Borel subgroup corresponding to ${\mathfrak{b}}\subset \g$, and $B^{\rm ad}$ be the image of $B$ in $G^{\rm ad}$.  

\begin{corollary} 
\label{AutRepuqb}
One has $\Aut(\Rep \mathfrak{u}_q(\mathfrak{b}))\cong \Gamma\ltimes B^{\rm ad}$. 
\end{corollary} 

\begin{proof}
Let $\C$ be a finite tensor category. Let $\Inv(\C)$ denote the group of isomorphism classes of invertible objects of $\C$.
There is an exact sequence 
$$ 
\Inv(\mathcal{Z}(\C))\to \Inv(\C)\to \Aut(\C)\to \BrPic(\C),
$$
see \cite{GP, MN}, where the first map is induced by the forgetful functor $\mathcal{Z}(\C)\to \C$, 
the second one sends an invertible object $\chi$ to the conjugation functor $X\mapsto \chi\otimes X\otimes \chi^{-1}$, and 
the third one is given by  $\theta(F)=\C$ with the usual left action of $\C$ and the right action of $\C$ twisted by $F$. 

Now take $\C=\Rep \mathfrak{u}_q(\mathfrak{b})$. Then the above exact sequence takes the form
$$
E\to E\to \Aut(\C)\to \BrPic(\C),
$$
where the first map is the identity. Thus, the map $\theta: \Aut(\C)\to \BrPic(\C)$ is injective. Hence, by Proposition \ref{BP}(i),
$\Aut(\C)\subset (\Gamma\ltimes G^{\rm ad})\times O(E,\mathbf{q})$. 

It is clear that $\Aut(\C)$ contains the subgroup $\Gamma_{\rm diag}\ltimes B^{\rm ad}\subset  (\Gamma\ltimes G^{\rm ad})\times O(E,\mathbf{q})$, 
where $\Gamma_{\rm diag}\subset \Gamma\times O(E,\mathbf{q})$ is the diagonal copy of $\Gamma$. 
Also, one shows similarly to the proof of Proposition \ref{conncomid} (using the results of \cite{GK}) that 
$\Lie\Aut(\C)={\mathfrak{b}}$, hence $$\Aut(\C)_0=B^{\rm ad}\subset G^{\rm ad}.$$ Thus, $\Aut(\C)$ must normalize $B^{\rm ad}$, hence
$\Aut(\C)=\Gamma'\ltimes B^{\rm ad}$, where $$\Gamma_{\rm diag}\subset \Gamma'\subset \Gamma\times O(E,\mathbf{q}).$$ 

It remains to show that $\Gamma'=\Gamma_{\rm diag}$. 
Let $F\in \Aut(\C)$, and consider the action of $F$ on the invertible objects $\Inv(\C)=E$. 
First of all, $F$ must permute the objects $g_i\in E$ corresponding to the roots $\alpha_i\in Q$, since they are the only invertible objects 
which have a nontrivial $\Ext^1$ with the unit object. Also, $\Ext^2(\bold 1,g_ig_j)=0$ if and only if $i$ is connected to $j$ in the Dynkin diagram of $\g$ (as this is exactly the case when there is no quadratic relation between $e_i$ and $e_j$). Thus, the permutation of $g_i$ induced by $F$ is implemented by an element of $\Gamma$. Hence, composing $F$ with an element of 
$\Gamma_{\rm diag}$ if needed, we may assume that $F$ acts trivially on $E$. Then $F\in G^{\rm ad}\cap \Aut(\C)=B^{\rm ad}$. 
This implies the required statement. 
\end{proof} 

\begin{remark}
Corollary~\ref{AutRepuqb} allows one to describe tensor autoequivalences of $\Rep \mathfrak{u}_q(\mathfrak{g})$ in terms of induction.
Namely,  given a tensor category $\C$, its  indecomposable exact  module categories are in bijection with  Lagrangian algebras in $\Z(\C)$.
This bijection is given by 
\begin{equation}
\label{modLagrcorr}
\M \mapsto I_\M(\be), 
\end{equation}
where $I_\M: \C^*_\M \to \Z(\C)$ is the right adjoint to the forgetful functor $\Z(\C)\cong \Z(\C^*_\M) \to  \C^*_\M$. 
As usual,  $\C^*_\M$ denotes the dual tensor category of $\C$ with respect to $\M$.  Furthermore, correspondence~\eqref{modLagrcorr}
is equivariant with respect to the isomorphism $\BrPic(\C) \cong \Aut^{\rm br}(\Z(\C))$. Here the group $\BrPic(\C)$  (respectively, 
$\Aut^{\rm br}(\Z(\C))$) acts on the set of module categories (respectively, Lagrangian algebras) in an obvious way.
The stabilizer of $\M$ in $\Aut^{\rm br}(\Z(\C))$
is the subgroup  of autoequivalences induced  from $\Aut(\C^*_\M)$ and the orbit of $\M$
consists of module categories $\N$ such that $\C^*_\M \cong \C^*_\N$ \cite{MN}. 

In our situation $\C = \Rep \mathfrak{u}_q(\mathfrak{b})$, the induction    
$\Aut(\Rep \mathfrak{u}_q(\mathfrak{b}))\to \Aut^{\rm br}(\Rep \mathfrak{u}_q(\mathfrak{g}))$ is injective,
and $\Aut^{\rm br}(\Rep \mathfrak{u}_q(\mathfrak{g}))/ \Aut(\Rep \mathfrak{u}_q(\mathfrak{b}))$ is identified
with the flag variety $G/B = G^{\rm ad}/B^{\rm ad}$. Point stabilizers are identified with images of inductions:
\[
\Aut( (\Rep \mathfrak{u}_q(\mathfrak{b}))^*_\M) \to \Aut^{\rm br}(\Rep \mathfrak{u}_q(\mathfrak{g}))
\]
taken over module categories  $\M$  such that  $(\Rep \mathfrak{u}_q(\mathfrak{b}))^*_\M \cong \Rep \mathfrak{u}_q(\mathfrak{b})$.
Since $G^{\rm ad}$ coincides with the union of its Borel subgroups and all Borel subgroups are conjugate,
we conclude that every (braided) tensor autoequivalence of $\Rep \mathfrak{u}_q(\mathfrak{g}) $ is induced from
a tensor autoequivalence of a copy of $\Rep \mathfrak{u}_q(\mathfrak{b})$ (i.e., from a central tensor functor
$\Rep \mathfrak{u}_q(\mathfrak{g}) \to \Rep \mathfrak{u}_q(\mathfrak{b})$).
\end{remark}

\begin{remark}

The induction homomorphism
$\Aut( \Rep \mathfrak{u}_q(\mathfrak{b})) \to \Aut^{\rm br}(\Rep \mathfrak{u}_q(\mathfrak{g}))$
and construction of Weyl  reflections in  $\Aut^{\rm br}(\Rep \mathfrak{u}_q(\mathfrak{g}))$
are discussed in \cite{LP}.
\end{remark}

\subsection{Twists for $\mathfrak{u}_q(\g)$}

It is an interesting problem to classify twists for $\mathfrak{u}_q(\g)$ up to gauge transformations, 
i.e., categorically speaking, to classify fiber functors $F: \Rep\mathfrak{u}_q(\g)\to \Vec$ up to isomorphism. 
By the results of \cite{EK1,EK2}, the answer to a similar question for the quantized universal enveloping algebra $U_\hbar(\g)$ 
is given in terms of Belavin-Drinfeld triples (see e.g., \cite{KKSP}). On the other hand, twists for 
$\mathfrak{u}_q(\g)$ associated to Belavin-Drinfeld triples were worked out in \cite{Ne} 
following the method of \cite{EN} and \cite{ESS}. Let us call them {\it Belavin-Drinfeld twists}, and call the corresponding fiber functors {\it Belavin-Drinfeld functors}. 

\begin{question}(see also \cite{Ne}, Question 9.5) 
Is any fiber functor on $\Rep\mathfrak{u}_q(\g)$ a composition of a Belavin-Drinfeld functor with a tensor autoequivalence of $\Rep\mathfrak{u}_q(\g)$? 
In other words, is any twist for $\mathfrak{u}_q(\g)$ gauge equivalent to a composition of a Belavin-Drinfeld twist with one coming from a twisted automorphism of 
$\mathfrak{u}_q(\g)$? 
\end{question}

The answer is positive for $\g=\mathfrak{sl}_2$ by \cite[Proposition 8.11]{Mo}. 
In this case there are no nontrivial Belavin-Drinfeld functors, so every fiber functor is the composition 
of the standard one with a tensor autoequivalence, and tensor autioequivalences form the group $SL_2(k)$. 



\begin{thebibliography}{A} 

\bibitem[AGP]{AGP} I.~Angiono, C.~Galindo, M.~Pereira,
{\em De-equivariantization of Hopf algebras}, 
Algebras and Representation Theory \textbf{17}  (2014), no.~1, 161-180.

\bibitem[AG]{AG} S.~Arkhipov,  D.~Gaitsgory,
{\em Another realization of the category of modules over the small quantum group},
Adv.\ Math.\ \textbf{173} (2003), no.~1, 114-143.

\bibitem[Bi1]{Bi1} J.~Bichon, 
{\em The representation category of the quantum group of a non-degenerate bilinear form}, 
Comm.\ Algebra \textbf{31} (2003), no.~10, 4831-4851.

\bibitem[Bi2]{Bi2} J.~Bichon,
 {\em The group of bi-Galois objects over the coordinate algebra of the Frobenius-Lusztig kernel of $SL(2)$}, 
 Glasg.\ Math.\ J.\ \textbf{58} (2016), no.\ 3, 727--738.

\bibitem[BC]{BC} J.~Bichon, G.~Carnovale,
{\em Lazy cohomology: An analogue of the Schur multiplier for arbitrary Hopf algebras},
J.\ Pure and Applied Algebra \textbf{204} (2006), no.~3, 627-665.

\bibitem[BKa]{BKa} J.~Bichon and C.~Kassel, 
{\em The lazy homology of a Hopf algebra}, 
J.\ Algebra \textbf{323} (2010), no.~9, 2556--2590. 

\bibitem[BG]{BG} K.~Brown, K.~Goodearl, 
{\em Lectures on algebraic quantum groups},
Birkh\"auser  (2012).

\bibitem[DL]{DL} C.~De Concini, V.~ Lyubashenko, 
{\em Quantum function algebra at roots of 1},
Adv.\ Math.\ \textbf{108} (1994), no.\ 2, 205--262.

\bibitem [DN]{DN}  A.~Davydov, D.~Nikshych, 
{\em The Picard crossed module of a braided tensor category},
Algebra and Number Theory, \textbf{7} (2013), no.\ 6, 1365--1403.

\bibitem[Da1]{da} A. Davydov, {\em Twisting of monoidal structures}, Preprint of MPI (1995), MPI/95-123, arXiv:q-alg/9703001.

\bibitem[Da2]{Da} A. Davydov, {\em Twisted automorphisms of Hopf algebras}, in Noncommutative structures in Mathematics and Physics, Koninklijke Vlaamse Academie Van Belgie Voor Wetenschappen en Kunsten (2010) 103--130, arXiv:0708.2757.

\bibitem[Da3]{Da3} A. Davydov, {\em Twisted derivations of Hopf algebras}, J.\ Pure and Applied Algebra  \textbf{217} (2013), 567-582.

\bibitem[DGNO]{DGNO} V.~Drinfeld, S.~Gelaki, D.~Nikshych, V.~Ostrik,
\textit{On braided fusion categories I}, Selecta Mathematica, \textbf{16} (2010), no.\ 1,  1--119.

\bibitem[DM]{DM} P.~Deligne, J.~Milne, 
{\em Tannakian categories},  Lecture Notes in Mathematics \textbf{900} (1982). 
 
\bibitem[EK1]{EK1} P.~Etingof, D.~Kazhdan, 
{\em Quantization of Lie bialgebras I},
Selecta Math.\ \textbf{2} (1996), no.~1, 1-41.

\bibitem[EK2]{EK2} P.~Etingof, D.~Kazhdan, 
{\em Quantization of Lie bialgebras II},
Selecta Math.\  \textbf{4} (1998), no.~2, 213-231.
 
\bibitem[Eis]{Eis} D.~Eisenbud, 
{\em Commutative algebra with a view towards algebraic geometry}, 
Vol.\ 150. Springer Science and  Business Media  (2013).


\bibitem[EGNO]{EGNO} P.~Etingof, S.~Gelaki, D.~Nikshych, V.~Ostrik, 
{\em  Tensor categories}, 
Mathematical Surveys and Monographs,
\textbf{205}, American Mathematical Society (2015).

\bibitem[ENO]{ENO}   P.~Etingof, D.~Nikshych, and V.~Ostrik.
{\em  Fusion categories and homotopy theory}, 
Quantum Topology,   \textbf{1} (2010), no.\ 3,  209-273.

\bibitem[EN]{EN} P.~Etingof, D.~Nikshych, 
{\em Dynamical quantum groups at roots of 1}, 
Duke Math.\ J.\ \textbf{108} (2001), no.~1, 135-168.

\bibitem[EO]{EO}  P. ~Etingof,  V. ~Ostrik,
{\em Finite  tensor  categories},  Moscow  Math.\  J.\ \textbf{4} (2004), 627-654.

\bibitem[ESS]{ESS} P.~Etingof, T.~Schedler, O.~Schiffmann,
{\em Explicit quantization of dynamical $r$-matrices for finite dimensional semisimple Lie algebras},
J.\ Amer.\ Math.\ Soc.\  \textbf{13} (2000), 595-609 .

\bibitem[FRT]{FRT} L.D.~Faddeev, N.Yu.~Reshetikhin, and L.A.~Takhtajan, 
{\em Quantization of Lie groups and Lie algebras}, 
in Algebraic Analysis, Vol. I (M.~Kashiwara and T.~Kawai, eds.), Academic
Press, Boston, 1988, 129-139.

\bibitem[FS]{FS} J.~Fuchs, C.~Schweigert,
{\em Symmetries and defects in three-dimensional topological field theory}, 
String-Math 2014, 21-40, Proc.\ Sympos.\ Pure Math. \textbf{93}, American Mathematical Society (2016).

\bibitem[GP]{GP}  C.~Galindo, J.~Plavnik,
{\em Tensor functors between Morita duals of fusion categories},
Letters in Math.\ Phys.\ \textbf{107} (2017), no.~3, 553-590.

\bibitem[GK]{GK} V.~Ginzburg, S.~Kumar, 
{\em Cohomology of quantum groups at roots of unity}, 
Duke Math.\ J.\ \textbf{69} (1993), no.\ 1, 179--198.

\bibitem[GKa]{GKa} P.~Guillot, C.~Kassel, 
{\em Cohomology of invariant Drinfeld twists on group algebras},
Int.\ Math.\  Res.\  Notices (2009), 1894-1939.

\bibitem[Gri]{Gri} R.I.~Grigorchuk, 
{\em On the system of defining relations and
the Schur multiplier of periodic groups generated by finite automata},  
in Groups St. Andrews 1997 in Bath, I, volume 260 of
London Math. Soc. Lecture Note Ser., 290-317. Cambridge Univ. Press, Cambridge, 1999.

\bibitem[Ha]{Ha} T.~Hayashi, 
{\em Quantum deformation of classical groups}, 
Publ.\ Research Inst.\ Math.\ Sci.\ (Kyoto) \textbf{28} (1992), 57-81.

\bibitem[I]{I} R.~Iglesias, 
{\em Bitableaux bases of the quantum coordinate algebra of a semisimple group}, 
J.\ Algebra \textbf{301} (2006), no.~1, 308-336.

\bibitem[Is]{Is}  J.R.~Isbell, {\em Epimorphisms and  dominions}, Proc.\ Conf.\ Categorical Algebra (La  Jolla,
Calif., 1965), Springer, New York, 1966, 232--246.  

\bibitem[KW]{KW}  D.~Kazhdan, H.~Wenzl,
{\em Reconstructing monoidal categories}, 
I.M.~Gelfand Seminar, Amer.\ Math.\ Soc., Providence, RI, 1993, 111-136.

\bibitem[KS]{KS} L.~Korogodsky, Y.~Soibelman, 
{\em Algebras of functions on quantum groups: Part  I}, 
Mathematical Surveys and Monographs,
\textbf{56}, American Mathematical Society (1998).

\bibitem[KKSP]{KKSP}
B.~Kadets, E.~Karolinsky, A.~Stolin,  and I.~Pop, 
{\em Classification of quantum groups and Belavin-Drinfeld cohomologies},
Comm.\ Mathematical Physics \textbf{344} (2016), no.~1, 1-24.

\bibitem[LP]{LP} S.~Lentner,\, J.~Priel,
{\em Three natural subgroups of the Brauer-Picard group of a Hopf algebra with applications}
arXiv:1702.05133 [math.QA].

\bibitem[Lu1]{Lu1} G.~Lusztig,
{\em Finite dimensional Hopf algebras arising from quantized universal enveloping algebras},
J.\ Amer.\ Math.\ Soc.\  \textbf{3} (1990), no.~1, 257-296.

\bibitem[Lu2]{Lu2} G.~Lusztig,
{\em Modular representations and quantum groups}, Classical groups and related topics:
Proceedings of a conference in honor of L.K. Hua, Contemporary Mathematics \textbf{82}, Amer.\
Math.\ Soc. (1989),  59-77. 

\bibitem[Lu3]{Lu3} G.~Lusztig, 
{\em Study of a $\mathbf{Z}$-form of the coordinate ring of a reductive group}, 
J.\ Amer.\ Math.\ Soc.\  \textbf{22} (2009), no. ~3, 739-769. 

\bibitem[MN]{MN}  I.~Marshall,\, D.~Nikshych,
{\em On the Brauer-Picard groups of fusion categories}, 
arXiv:1603.04318 [math.QA].

\bibitem[Mo]{Mo} M.~Mombelli, 
{\em Module categories over pointed Hopf algebras},
Mathematische Zeitschrift \textbf{266} (2010), no.~2, 319-344.

\bibitem[Ne]{Ne} C.~Negron, 
{\em Small quantum groups associated to Belavin-Drinfeld triples}, 
arXiv:1701.00283 [math.QA]. 
        
\bibitem[NT1]{NT1} S.~Neshveyev, L.~Tuset, 
{\em Symmetric invariant cocycles on the duals of q-deformations},  
Adv.\ Math.\ \textbf{227} (2011), no.~1, 146--169. 
  
\bibitem[NT2]{NT2} S.~Neshveyev, L.~Tuset, 
{\em Autoequivalences of the tensor category of  $U_q(\g)$-modules},  
Int.\ Math.\ Res.\ Notices  (2012), no.~15, 3498-3508. 

\bibitem[RTF]{RTF} N.Yu.~Reshetikhin, L.A.~Takhtajan, and L.D.~Faddeev, 
{\em Quantization of Lie groups and Lie algebras}, 
Leningrad Math.\ J.\ \textbf{1} (1990), 193-225.

\bibitem[Ta]{Ta} M.~Takeuchi, 
{\em Quantum orthogonal and symplectic groups and their embedding into quantum GL}, 
Proc.\ Japan Acad., Ser. A Math.\ Sci.\ \textbf{65} (1989), 55-58.

\bibitem[T]{T} V.~Turaev,
{\em Quantum invariants of knots and $3$-manifolds},
W.\ de Gruyter (1994).

\bibitem[TW]{TW} I.~Tuba, H.~Wenzl, 
{\em On braided tensor categories of type BCD}, 
Journal f\"ur die reine und angewandte Mathematik \textbf{581} (2005), 31-69.

\bibitem[Y]{Y} D.~Yetter.
{\em Braided deformations of monoidal categories and Vassiliev invariants},
In: ``Higher category theory", \textit{Contemp.\ Math.\ AMS} {\bf 230}  (1998), 117--134.
   
\end{thebibliography}
\end{document}